\documentclass[12pt,leqno]{amsart}
\usepackage{amsmath, amssymb, amscd, amsfonts, xypic,   stmaryrd, turnstile, mathrsfs, eucal, color}

\usepackage
{hyperref}

\newtheorem{theorem}{Theorem}[section]
\newtheorem{lemma}[theorem]{Lemma}
\newtheorem{proposition}[theorem]{Proposition}
\newtheorem{corollary}[theorem]{Corollary}

\theoremstyle{definition}
\newtheorem{definition}[theorem]{Definition}
\newtheorem{example}[theorem]{Example}
\newtheorem{remark}[theorem]{Remark}

\begin{document}

\title[Pointwise minimal extensions]{ Pointwise minimal extensions}

\author[G. Picavet and M. Picavet]{Gabriel Picavet and Martine Picavet-L'Hermitte}
\address{Universit\'e Blaise Pascal \\
Laboratoire de Math\'ematiques\\ UMR6620 CNRS  \\ 24, avenue des Landais\\
BP 80026 \\ 63177 Aubi\`ere CEDEX \\ France}

\email{Gabriel.Picavet@math.univ-bpclermont.fr}
\email{picavet.gm(at)wanadoo.fr}

\begin{abstract}  We characterize pointwise minimal extensions of rings, introduced by P.-J. Cahen, D. E. Dobbs and  T. G. Lucas in a special context. We also define and characterize  pointwise minimal pairs of rings and co-pointwise minimal extensions.  We examine the links of the above notions with   lattices and their atoms.

\end{abstract}


\keywords  {FIP, FCP, minimal extension, integral extension, support of a module,  lattice, algebraic field extension, pointwise minimal extension, geometric lattices}

\maketitle

\section{Introduction and Notation}

We consider the category of commutative and unital rings and its epimorphisms. A {\it local} ring is here what is called elsewhere a quasi-local ring.  As usual, Spec$(R)$ and Max$(R)$ are the sets of prime and maximal ideals of a ring $R$.  The characteristic of an integral domain $k$ is denoted by $\mathrm{c}(k)$.
 Finally, $\subset$ denotes proper inclusion,  $|X|$ the cardinality of a set $X$ and $\mathbb{P}$  the set of all prime numbers.

 The conductor of a (ring) extension $R\subseteq S$ is denoted by $(R:S)$, the set of all $R$-subalgebras of $S$ by $[R,S]$ and the integral closure of $R$ in $S$ by $\overline R$.  Any writing $[R,S]$ is relative to some extension $R\subseteq S$. Clearly, $([R,S], \subseteq )$ is a lattice since  it is stable under the formation of arbitrary intersections (meets) and compositums (joins). If $[R,S]$ has some property $\mathcal P$ of lattices, we  say that $R\subseteq S$ has the  property $\mathcal P$.

 An extension $R\subseteq S$ is called an {\it  afffine pair (or strongly affine)} if each $R$-subalgebra of $S$ is  of finite type.   We say that an extension $R\subseteq S$ is finite if the $R$-module $S$ is finitely generated.

 The extension $R\subseteq S$ is said to have FIP (or is called an FIP extension) (for the ``finitely many intermediate algebras property") if $[R,S]$ is finite. A {\it chain} of $R$-subalgebras of $S$ is a set of elements of $[R,S]$ that are pairwise comparable with respect to inclusion. We say that an  extension $R\subseteq S$ has FCP (or is called an FCP extension)  (for the ``finite chain property") if each chain in $[R,S]$ is finite.  Dobbs and the authors characterized FCP and FIP extensions \cite{DPP2}. A mighty tool is the concept  of minimal (ring) extensions,  introduced by Ferrand-Olivier \cite{FO}. Recall that an extension $R\subset S$ is called {\it minimal} if $[R, S]=\{R,S\}$. The key connection between the above ideas is that if $R\subseteq S$ has FCP, then any maximal (necessarily finite) chain of $R$-subalgebras of $S$, $R=R_0\subset R_1\subset\cdots\subset R_{n-1}\subset R_n=S$, with {\it length} $n <\infty$, results from juxtaposing $n$ minimal extensions $R_i\subset R_{i+1},\ 0\leq i\leq n-1$. For any extension $R\subseteq S$, the {\it length} of $[R,S]$, denoted by $\ell[R,S]$, is the supremum of the lengths of chains of $R$-subalgebras of $S$. It should be noted that if $R\subseteq S$ has FCP, then there {\it does} exist some maximal chain of $R$-subalgebras of $S$ with length $\ell[R,S]$ \cite[Theorem 4.11]{DPP3}.

We come now to the subject of the paper. In \cite{CDL}, Cahen, Dobbs and Lucas call an extension $R\subset S$  {\it pointwise minimal} if $R\subset R[t]$ is a minimal extension for each $t\in S\setminus R$. We study such extensions in Section 3 and a special type of these extensions: a ring extension $R\subset S$ is called a {\it pointwise minimal pair} if $T\subset S$ is pointwise minimal  for each $T\in[R,S]\setminus\{S\}$.

 Clearly, the following implications hold: minimal extension $\Rightarrow$ pointwise minimal pair $\Rightarrow$  pointwise minimal  extension.
 We also define a dual notion in Section 3;  that is, co-pointwise minimal extensions. 

\begin{theorem}\label{1.2}\cite[Th\'eor\`eme 2.2]{FO} A minimal extension $A\subset B$ 
 defines a monogenic algebra which is either finite, or a flat epimorphism and these conditions are mutually exclusive. 
\end{theorem}

Results on flat epimorphisms are summed up  in   \cite[Scholium A]{Pic 5}.
Knebusch and Zhang defined Pr\"ufer extensions in \cite{KZ}. Among a lot of characterizations an extension $R\subseteq S$ is Pr\"ufer if and only if $R\subseteq T$ is  a flat  epimorphism for each $T \in [R,S]$. 

In \cite{Pic 5}, we called an extension which is a minimal flat epimorphism, a {\it Pr\"ufer minimal} extension. From now on, we use this terminology.  A  pointwise minimal extension  is either  integral or integrally closed, in which case it is Pr\"ufer minimal. It follows that our study can be reduced to the case of integral extensions.  A pointwise minimal  extension $R \subset S$ has a crucial ideal $M$ {\it i.e.}   the support of the $R$-module $S/R$ is $\{M\}$ and $M$ is necessarily a maximal ideal.  In case $R\subset S$ is integral and pointwise minimal, its crucial ideal is $(R:S)$.  Those statements (appearing in \cite{CDL} in a special context) are proved in Section 2  and Section 3 and  are essential in this paper.   We will also need  the canonical decomposition of an integral extension $R \subseteq {}_S^+R\subseteq {}_S^tR \ \subseteq S$, where ${}_S^+R$ and $ {}_S^tR$ are the seminormalization and the $t$-closure of $R$ in $S$ (see Section 2 for the details).  Our strategy is as follows. We first suppose that $R$ is a field in Section 4   and then we consider in Section 5  an integral extension, whose conductor is a maximal ideal, much more easy to handle than a crucial ideal. Surprisingly, we are able to classify pointwise minimal integral extensions $R \subset S$: either the seminormalization and the $t$-closure coincide  or  $R = {}_S^+R$ and  ${}_S^tR = S$.  Then in Section 5 we get a complete characterization of pointwise minimal  integral extensions and pairs that are not minimal, while  co-pointwise minimal extensions are characterized  in Section 3 as pointwise minimal pairs of length $2$ and more precisely at the end of Section 5.   Naturally, we consider the special case of FCP and FIP extensions.  Section 6 is concerned with examples and applications. In particular, we consider Nagata extensions. To end, Section 7 deals with properties of lattices and their atoms (in our context, they are minimal extensions),  linked to the above notions. In particular, finitely geometric lattices are involved.

\section {Some  useful  results and recalls}

We need to give some notation and definitions.

 If $I$ an ideal of  a ring $R$, we denote by $\mathrm V_R(I)$  or ($\mathrm V(I)$) the closed subset 
$\{ P \in  \mathrm{Spec}(R) \mid I \subseteq P\}$,  by $\mathrm D_R(I)$ its complement and by $\sqrt[R]{I}$ the radical of $I$ in $R$.
The support of an $R$-module $E$ is $\mathrm{Supp}_R(E):=\{P\in\mathrm{Spec }(R)\mid E_P\neq 0\}$, and $\mathrm{MSupp}_R(E):=\mathrm{Supp}_R(E)\cap\mathrm{Max}(R)$.   If $R\subseteq S$ is a ring extension and $P\in\mathrm{Spec}(R)$, then $S_P$ is both the localization $S_{R\setminus P}$ as a ring and the localization at $P$ of the $R$-module $S$. For a ring morphism $f: R \to S$ and $Q \in \mathrm{Spec}(S)$, we denote by $\kappa(P) \to \kappa(Q)$ the residual extension, where $P=f^{-1}(Q)$.

 \begin{definition}\label{1.19}  We say that  an extension $R\subset S$ has  a a crucial ideal $\mathcal{C}(R,S):= M \in \mathrm{Spec}(R)$ if $\mathrm{Supp}_R(S/R)=\{M\}$ and in this case call the extension $M$-{\it crucial}. A crucial ideal needs to be maximal because a support is stable under specialization.  \end{definition}
 
 For example,  a minimal extension has a crucial ideal \cite[Th\'eor\`eme 2.2]{FO}. We will show later that a pointwise minimal extension has also a crucial ideal.
We begin by proving some results on crucial ideals. 
\subsection{ Crucial ideals  and radicial extensions}
 
 In the sequel,  $ \{R_\alpha \mid \alpha \in I\}$  is the family of all  finite extensions 
$R \subset R_\alpha $ with $R_\alpha \in [R,S]$   and  conductor $C_\alpha$.
 
 \begin{proposition}\label{1.191}  Let $R \subset S$ be an extension, with conductor $C$. The following statements hold:
 
 (1) If $R\subset S$ is $M$-crucial, then $C \subseteq M$.
 
 (2)  If  $R\subset S$ is integral, then $R\subset S$  has a crucial  ideal if and only if $\sqrt C \in \mathrm{Max}(R)$,  and then  $\mathcal C (R,S)= \sqrt C$.
\end{proposition}

\begin{proof}   

(1) If the extension is $M$-crucial, suppose that there is some $x \in C\setminus M$, then it is easily seen that $R_M=S_M$, a contradiction.

 (2) For  $M \in \mathrm{Spec}(R)$,  observe that $M$ is a crucial ideal of  $R \subset S$    if and only if  $M$ is a crucial ideal of  each $R\subset R_\alpha$. Then it is enough to use the following facts: $\mathrm{Supp}(R_\alpha /R) = \mathrm{V}(C_\alpha)$ and $C = \cap [C_\alpha \mid \alpha \in I]$.  
\end{proof}

An   $M$-crucial integral extension has the following properties.  If $Q\in \mathrm{Spec}(S)$ is  lying over $P \in \mathrm{Spec}(R)$, then 
$R_P\to S_P$ and $\kappa(P) \to  \kappa (Q)$ are isomorphisms 
 if $P\neq  M$. Moreover, 
  $\kappa (M)  \to \kappa (Q)$ is of the form $R/M \to S/ Q$ for each  $Q\in \mathrm{Max}(S)$  lying over $M$. Observe also that for an integral extension $R\subset S$, with $C:=(R:S)$, there is a bijection $\mathrm{D}_S(C)\to\mathrm{D}_R(C)$ defined by $Q\mapsto Q\cap R$.
  
 The Nagata ring of a ring $R$ is  $R(X):= R[X]_{\Sigma}$, where $\Sigma$ is the multiplicatively closed  subset  of polynomials whose contents are $R$. 
We compute the crucial ideal of a Nagata extension $R(X) \subset S(X)$ when  $R \subset S $ is $M$-crucial.  Recall that if $R\subset S$ is integral,  then $S(X) \cong  R(X)\otimes_R S$ \cite[Lemma 3.1]{DPP3}. The same property holds if $R\subset S$ is a flat epimorphism since the surjective map $R(X)\otimes_R S \to S(X)$ is injective. Indeed, $R(X) \to R(X) \otimes_R S$ is a flat epimorphism and  $R(X) \to  R(X)\otimes_R S \to  S(X)$ is injective \cite[Scholium A(3)]{Pic 5}.
  
\begin{lemma}\label{7.172}   Let $R\subset S$ be an $M$-crucial  extension such that $S(X) \cong R(X)\bigotimes_RS$ (for example, if $R\subset S$ is integral or a flat epimorphism). Then $R(X) \subset S(X)$ is $MR(X)$-crucial.  
\end{lemma}

\begin{proof} The  extension $g:R\to R(X)$ is  faithfully flat  and $\mathrm{Supp}(S/R)=\{M\}$. Let $Q\in\mathrm{Supp}(S(X)/R(X))$. Applying \cite[Proposition 2.4 (b)]{DPP3}, we get that $Q\in ({}^{a}g)^{-1}(\mathrm{Supp}(S/R))$, so that ${}^{a}g(Q)=Q\cap R\in\mathrm{Supp}(S/R)=\{M\}$, giving $M=Q\cap R$. It follows that $M\subseteq Q$, which implies $MR(X)\subseteq Q$ and then $Q=MR(X)$ since $MR(X)\in\mathrm{Max}(R(X))$. Therefore,  $\mathrm{Supp}(S(X)/R(X))=\{MR(X)\}$.
\end{proof}

  We will call in this paper {\it radicial} any purely inseparable field extension, in order to have a terminology consistent with radicial (radiciel in French) extensions of rings. Recall that a ring morphism  $R\to S$ is called radicial  if $ \mathrm{Spec}(R'\otimes_R S) \to \mathrm{Spec}(R')$ is injective for any base change $R \to R'$. A ring extension $R\to S$ is radicial if and only if $\mathrm{Spec}(S) \to \mathrm{Spec}(R)$ is injective and its  residual extensions  are radicial \cite[Proposition 3.7.1]{EGA}. Also a radicial extension $K\subset L$ of fields is said to  have height one if $x^p \in K$ for each $x\in L$, where $p:=\mathrm{c}(K) \in \mathbb P$. We  say that a ring extension $K\subset S$, where $K$ is a field,  is {\it radicial of height one}  if $\mathrm{c}(K) =p \in \mathbb P$ and $x^p \in K$ for each $x\in S$. Indeed, such an extension is radicial, as it is easily seen.

An $M$-crucial extension $R\subset S$, such that $M =(R:S)$, 
 is called a {\it height one radicial extension} if so is $R/M \subset S/M$. Such an extension is again radicial, by the above  considerations. 
 
 \subsection{Results on  minimal extensions} 

  There are three types of minimal integral extensions, characterized by the following theorem, from  the fundamental lemma of Ferrand-Olivier.

\begin{lemma}\label{1.3} \cite [Lemme 1.2]{FO} An extension $K\subset A$, where $K$ is a field, is  minimal  if and only if one of the following conditions is satisfied, in which case $K\subset A$ is finite:
\begin{enumerate}
\item   $A$ is a field and $K\to A$ is a minimal field extension.

\item   $f$ is the diagonal morphism $K\to K\times K$.

\item  $f$ is the canonical morphism $K\to D_K(K)=K[X]/(X^2)$.
\end{enumerate}

\end{lemma}

\begin{lemma}\label{1.3.1}

The  following statement hold:
 
\begin{enumerate}
 
\item  Minimal field extensions coincide with minimal ring extensions between fields. 
 
 \item  A minimal field extension is either separable or radicial. 
 
\item A radicial  field  extension $K\subset L$, with $\mathrm{c}(K) =p \in \mathbb{P}$, is minimal if and only if  $L= K[x] $ where $x^p \in K$, and if and only if $K\subset L$ is monogenic and radicial of height one. 

\end{enumerate}
\end{lemma}

\begin{proof} \cite[p. 371]{Pic 7}.
\end{proof}

\begin{theorem}\label{1.4} \cite [Theorem 2.2]{DPP2} 
An extension $R\subset T$ is minimal and finite if and only if $M:= (R:T)\in   \mathrm{Max}(R)$ and one of the following three conditions holds:

\noindent (a) {\bf inert case}: $M\in\mathrm{Max}(T)$ and $R/M\to T/M$ is a minimal field extension.

\noindent (b) {\bf decomposed case}: There exist $M_1,M_2\in\mathrm{Max}(T)$ such that $M= M _1\cap M_2$ and the natural maps $R/M\to T/M_1$ and $R/M\to T/M_2$ are both isomorphisms.

\noindent (c) {\bf ramified case}: There exists $M'\in\mathrm{Max}(T)$ such that ${M'}^2 \subseteq M\subset M',\  [T/M:R/M]=2$, and the natural map $R/M\to T/M'$ is an isomorphism.
\end{theorem}

We give here a lemma  used in earlier papers and introduce FMC extensions. An extension $R\subset S$ is said to have FMC (for a ``finite maximal chain'' property) if there is a finite maximal chain of extensions going from $R$ to $S$. Minimal and FCP extensions have FMC.  

\begin{lemma} \label{1.5}  Let $R\subset S$ be an  extension and $T,U\in[R,S]$ such that $R\subset T$ is a finite minimal  extension and $R\subset U$ is a Pr\"ufer minimal extension. Then, $\mathcal{C}(R,T)\neq\mathcal{C}(R,U)$, so that $R$ is not a local ring.
\end{lemma}

\begin{proof} Assume that $\mathcal{C}(R,T)=\mathcal{C}(R,U)$ and set $M:=\mathcal{C}(R,T)=(R:T)=\mathcal{C}(R,U)\in\mathrm{Max}(R)$. Then, $MT=M$ and  $MU=U$ because $R\subset U$ is a Pr\"ufer minimal extension. It follows that $MUT=UT=MTU=MU=U$, a contradiction.
\end{proof}

{\begin{lemma} \label{1.6}  Let $R\subset S$ be an FMC extension. If $M\in\mathrm{MSupp}(S/R)$, there exists $T\in[R,S]$ such that $R\subset T$ is minimal with $\mathcal{C}(R,T)=M$. 
\end{lemma}

\begin{proof} Let $\{R_i\}_{i=1}^n$ be a finite maximal chain such that $R_0:=R$ and $R_n:=S$. If $\mathcal{C}(R,R_1)=M$, then, $T=R_1$. So, assume that $M\neq \mathcal{C}(R,R_1)$. Let $k\in\{1,\ldots n-1\}$ be the least integer $i$ such that $M=\mathcal{C}(R_i,R_{i+1})\cap R$ \cite[Corollary 3.2]{DPP2}. For each $i<k$, we have $M\neq\mathcal{C}(R_i,R_{i+1})\cap R$, so that $M\in\mathrm{Max}(R)\setminus \mathrm{MSupp}(R_k/R)$. In view of \cite[Lemma 1.10]{Pic 3}, there exists $T\in[R,R_{k+1}]$ such that $R\subset T$ is minimal (of the same type as $R_k\subset R_{k+1}$) with $\mathcal{C}(R,T)=M$. 
\end{proof}

\subsection{The canonical decomposition of an integral extension}

\begin{definition}\label{7.90.0} An integral extension $R\subseteq S$ is called {\it infra-integral} \cite{Pic 2} (resp$.$; {\it subintegral} \cite{S}) if all its residual extensions  are isomorphisms (resp$.$; and the spectral map $\mathrm {Spec}(S)\to\mathrm{Spec}(R)$ is bijective). An extension $R\subseteq S$ is called {\it t-closed} (cf. \cite{Pic 2}) if the relations $b\in S,\ r\in R,\ b^2-rb\in R,\ b^3-rb^2\in R$ imply $b\in R$. The $t$-{\it closure} ${}_S^tR$ of $R$ in $S$ is the smallest element $B\in [R,S]$  such that $B\subseteq S$ is t-closed and the greatest element  $B'\in [R,S]$ such that $R\subseteq B'$ is infra-integral. An extension $R\subseteq S$ is called {\it seminormal} (cf. \cite{S}) if the relations $b\in S,\ b^2\in R,\ b^3\in R$ imply $b\in R$.  If $R\subseteq S$ is  seminormal, then $(R:S)$ is a radical ideal of $S$  and of $R$ (the proof goes back to Traverso). The {\it seminormalization} ${}_S^+R$ of $R$ in $S$ is the smallest element $B\in [R,S]$ such that $B\subseteq S$ is seminormal and the greatest  element $ B'\in[R,S]$ such that $R\subseteq B'$ is subintegral.

The canonical decomposition of an arbitrary ring extension $R\subset S$ is $R \subseteq {}_S^+R\subseteq {}_S^tR \subseteq \overline R \subseteq S$. 
\end{definition}

\begin{lemma}\label{7.90.1} Let $i: R\subset S$ be an integral extension such that ${}_S^+R={}_S^tR$ and $M:=(R:S) \in \mathrm{Max}(R)$. Set $T:={}_S^+R={}_S^tR$, then there is a unique maximal ideal $N:=\sqrt[S]{M} $ of $S$ lying over $M$, which is also the unique maximal ideal  of $T$ lying over $M$, so that $N=(T:S)$.
\end{lemma}

\begin{proof}  Since $R\subset T$ is subintegral, there is a unique maximal ideal $N$ of $T$ lying over $M$. Moreover, $M=(R:S)\subseteq (T:S)$. But $T\subseteq S$ is t-closed, and then seminormal. Then, $(T:S)$ is an intersection of prime ideals of $T$ by  \cite[Lemma 4.8]{DPP2}, which contain $M$, so that $(T:S)=N$.  Because $T/N\subset S/N$ is t-closed,  $N$ is a maximal ideal of $S$  \cite[Lemme 3.10]{Pic 1}, and the only maximal ideal  of $S$ lying over $M$, since it is equal to the only maximal ideal of $T$ lying over $M$. To get that $N= \sqrt[S]{M}$, observe that $i^{-1}(\{M\}) = \mathrm V_S(MS)$ and $MS =M$.
\end{proof}

We note  for further use that the classes of infra-integral   and subintegral  extensions are both stable under left or right divisions. The following proposition gives the link between the elements of the canonical decomposition and minimal extensions.

\begin{proposition}\label{7.90} \cite[Lemma 3.1]{Pic 4}  Let $R \subset S$ be an integral extension and  a tower of minimal extensions  $R=R_0\subset\cdots\subset R_i\subset\cdots\subset R _n= T $ in  $[R,S]$. Then, 

\begin{enumerate}
\item $R\subset T $ is subintegral if and only if each   $R_i\subset R_{i+1}$ is  ramified. 

\item $R\subset T$ is  infra-integral if and only if each  $R_i\subset R_{i+1}$ is either ramified or decomposed. 

\item $R\subset T$ is seminormal and infra-integral if and only if each  $R_i\subset R_{i+1}$ is decomposed. 

\item $R\subset T$ is seminormal  if and only if each  $R_i\subset R_{i+1}$ is either decomposed or t-closed. 

\item $R \subset T$ is t-closed if and only if  each $R_i\subset R_{i+1}$ is inert. 
\end{enumerate}
Moreover, if $R\subset S$ is subintegral, (resp.; infra-integral,  seminormal, t-closed), $R\subset T$ has the same property.
\end{proposition}

\begin{proof} We can suppose that $T= S$. \cite[Lemma 3.1]{Pic 4} asserts  that (2) and (5) holds.
 Now (1) is clear since  we deal with a bijective  spectral map. If $R\subset S$ is seminormal, 
  $(R:S)$ is a finite intersection of  maximal ideals of $S$ (resp. $R_{i+1}$)  by an easy generalization of \cite[Proposition 4.9]{DPP2}, giving (3)  and (4). 
\end{proof}

\section{General properties on pointwise minimal  extensions}

\subsection{First results on pointwise minimal extensions}

The next proposition generalizes to arbitrary extensions some results of \cite{CDL} gotten in the integral domains context. The proofs need only slight changes. 

\begin{proposition}\label{7.2.1}An extension $R\subset S$ is pointwise minimal if and only if $R\subset S$ is $M$-crucial  and $R_M\subset S_M$ is  pointwise minimal.
\end{proposition}

\begin{proof}  We generalize  \cite[Theorem 4.5]{CDL}.
Assume that  $R\subset S$ is $M$-crucial and   $R_M\subset S_ M$ is  pointwise minimal. Then  $R_N=S_N$ for each $N \in \mathrm{Spec}(R)\setminus \{M\} $. For any $x\in S\setminus R$, we get that $R_M\subset R_M[x/1]$ is  minimal because $x/1\not\in R_M$. Then, $R\subset R[x]$ is  minimal  \cite[Proposition 4.6]{DPPS} and $R\subset S$ is  pointwise minimal.

Conversely, assume that $R\subset S$ is  pointwise minimal. Let $x,y\in S\setminus R$, so that $R\subset R[x]$ and $R\subset R[y]$ are minimal, with respective crucial ideals $M$ and $N$. Assume that $N\neq M$. For any $P\in\mathrm{Spec}(R)\setminus\{M,N\}$, we have $R_P=R_P[x/1]=R_P[y/1]$. Moreover, $R_M=R_M[y/1]$ and $R_N=R_ N[x/1]$. From $M\neq N$, we infer that $R[x]\neq R[y]$. Set $z:=x+y$. Then, $z\not\in R[x]\cup R[y]$. In particular, $z\not\in R$, so that $R\subset R[z]$ is   minimal. For any $P\in\mathrm{Spec}(R)\setminus\{M,N\}$,  $ R_P=R_P[x/1]=R_P[y/1]$ holds, so that $R_P=R_P[z/1]$. Therefore, the crucial  ideal of $R\subset R[z]$ is either $M$ or $N$. If it is $M$, then, $R_N=R _N[z/1]=R_N[x/1]$, which yields $y/1=z/1-x/1\in R_N$, a contradiction. The same contradiction appears if $N$ is the crucial ideal of $R\subset R[z]$. Then, $M=N$ is the crucial  ideal of any minimal extension $R\subset R[x]$. In particular, $\mathrm{Supp}(S/R)=\{M\}$.
 \end{proof}
 
\begin{proposition}\label{7.2} Let $R\subset S$ be a ring extension. Then, 
\begin{enumerate}

\item \cite[Theorem 4.6]{CDL}  An integral  pointwise minimal extension $R\subset S$ is    $(R:S)$-crucial.
\item \cite[Proposition 4.7]{CDL} If $R$ and $S$ share an ideal $I$, then $R\subset S$ is a pointwise minimal extension (resp., pair) if and only if $R/I\subset S/I$ is a pointwise minimal extension (resp., pair).

\item \cite[Proposition 4.2]{CDL} If $R\subset S$ is  a pointwise minimal extension (resp., pair) and $T\in[R,S]\setminus\{R\}$ (resp., $T\subset T'$ a subextension of $R\subset S$), then $R\subset T$ (resp., $T\subset T'$) is a pointwise minimal extension (resp., pair).
\end{enumerate}
\end{proposition}

\begin{proof}
 For (1) and the parts of (2) and (3) related to pointwise minimal extensions, use the proofs of \cite[Theorem 4.6, Propositions 4.7 and 4.2]{CDL}. The proofs of (2) and (3) related to pointwise minimal pairs are obvious. 
  \end{proof}
  
  It may be asked if the trichotomy:  inert, ramified, decomposed of  finite minimal extensions  is still true for finite pointwise minimal extensions, in that sense:  if some $R[t]$  verifies some of these properties, all of them verify the same property. 
 The answer is no (see  Example~\ref{7.19}(5)).

\begin{corollary}\label{7.30}A pointwise minimal integral extension has FCP if and only if $R\subset S$ is finite, and  if and only if $\dim_{R/(R:S)}(S/(R:S))<\infty$.
\end{corollary}  

\begin{proof} In view of Proposition  \ref{7.2}(2), $M =(R:S) $ is a maximal ideal of $R$, so that $R/M$ is a field and $R\subset S$ has FCP if and only if $R \subset S$ is finite  \cite[Theorem 4.2]{DPP2}. The last equivalence is obvious.
\end{proof}

\begin{proposition}\label{7.6} A pointwise minimal extension $R\subset S$ is either integrally closed or integral.  
A pointwise minimal extension  has FCP if and only if it is an FMC extension.
\end{proposition}

\begin{proof} Assume that $R\subset S$ is neither integrally closed nor integral. Pick $x\in\overline R \setminus R$ and $y\in S\setminus\overline R$. The crucial ideal  $M$ of the extension is also the crucial ideal of $R\subset R[x]$ and $R\subset R[y]$. The first one is minimal integral and the second one is Pr\"ufer minimal  since $y\not\in\overline R$, which  contradicts  Lemma  \ref{1.5}. 
The last equivalence comes from   \cite[Theorem 4.2 and Theorem 6.3]{DPP2}. 
\end{proof} 

\begin{remark}\label{7.61} We can compare Proposition \ref{7.6} with \cite[Proposition 4.13]{CDL} of Cahen-Dobbs-Lucas. In their result, $R$ is a non-integrally closed local integral domain, $S$ is its quotient field and $R\subset \overline R$ is pointwise minimal. Then, each minimal overring of $R$ is contained in $\overline R$.
\end{remark}

The integrally closed case gives a simple result. 

\begin{proposition}\label{7.7} An integrally closed  extension $R\subset S$ is  pointwise minimal  if and only if it is  (Pr\"ufer) minimal.
\end{proposition}

\begin{proof} One implication is obvious. Now, assume that $R\subset S$ is  pointwise minimal.  For  $x\in S \setminus R$, $R\subset R[x]$ is minimal. But we have $R\subset R[x^2]\subseteq R[x]$, because $x$ is not integral over $R$; whence $R[x^2]= R[x]$. It follows  that $R\subset S$ is Pr\"ufer  \cite[Theorem 5.2, page 47]{KZ}. Therefore, $R\subset S$ has FCP  \cite[Proposition 1.3]{Pic 5} since $\mathrm{Supp}(S/R)=\{M\}$. We deduce from \cite[Theorem 6.3]{DPP2} that $R\subset S$ has FIP. To end, we get that $\ell[R,S]=1$, so  that $R\subset S$ is minimal by \cite[Proposition 6.12]{DPP2}.
\end{proof}

\begin{corollary}\label{7.8} A pointwise minimal extension is  either Pr\"ufer minimal or integral and these conditions are mutually exclusive.
 \end{corollary}

\begin{proof} Let $R\subset S$ be a pointwise minimal extension which is not integral. By Proposition \ref{7.6}, $R\subset S$ is integrally closed. If $R\subset S$ is integrally closed, then $R\subset S$ is minimal Pr\"ufer. 
\end{proof}
In order to characterize pointwise minimal integral extensions (resp.; pairs),   next results  will be useful. Recall that  $ \{R_\alpha \mid \alpha \in I\}$  is the family of all  finite extensions 
$R \subset R_\alpha $ with $R_\alpha \in [R,S]$   and  conductor $C_\alpha$.

\begin{lemma}\label{7.9}Let  $R\subset S$  be an integral extension, with conductor $C$. The following statements hold: 

\begin{enumerate}

 \item If $R/C$ is Artinian (for example, if $C$ is  a maximal ideal), the extensions  $R \subset R_{\alpha}$  have FCP and  $S =\cup [R_\alpha\mid \alpha \in I]$.  

\item  If the supremum of the lengths $\ell[R,R_{\alpha}]$ is a finite integer $n$, then $R\subset S$ has FCP and $\ell[R,S]=n$.
In case each $R\subset R_{\alpha}$ is minimal, then, $R\subset S$ is itself minimal.

\item If  in addition, $R\subset S$ is  pointwise minimal,  then it is $(R:S)$-crucial,   as well as each $R \subset R_\alpha$. 

\end{enumerate} 
\end{lemma}

\begin{proof} Since $R\subset S$ is integral, $S$ is the union  of the above upward directed family $\mathcal F:=\{R_{\alpha}\mid \alpha\in I\}$. 
 From $C \subseteq C_\alpha$,  we deduce that $R/C_\alpha$ is Artinian and then it is enough to use  \cite[Theorem 4.2]{DPP2}.

Assume that the supremum of the lengths $\ell[R,R_{\alpha}]$ is a finite integer $n$. Then, there clearly exists some $R_{\beta}$ such that $\ell[R,R_{\beta}]=n$.      Assume now that $R_{\beta}\neq S$, and let $x\in S\setminus R_{\beta}$. Then, $x$ is in some $R_{\gamma}$ and there exists some $R_{\delta}\in \mathcal F$ such that $R_{\beta},R_{\gamma}\subseteq R_{\delta}$ with $R_{\beta}\subset R_{\delta}$. This implies that $\ell[R,R_{\delta}]>n$, a contradiction. It follows that $S=R_{\beta}$, so that $R\subset S$ has FCP and $\ell[R,S]=n$. The last result is obvious.
\end{proof} 

 In Section 4, we  reduce our proofs to the case of fields.  The following is enlightening. Consider a field extension $K\subset L$, which is pointwise minimal. This  extension is necessarily algebraic, because  a flat epimorphism whose domain is a field is surjective \cite[Scholium A]{Pic 5}.  We will see later in Proposition \ref{7.13}  that $K\subset L$  is either minimal separable or radicial of height one.
Complexity of proofs relies heavily on this dichotomy.

\subsection{Co-pointwise minimal extensions}
  
  We are going to consider a property dual from the property of pointwise minimal extensions.

\begin{definition}\label{7.4} An extension $R\subset S$ is called a {\it co-pointwise minimal extension} if $R[x]\subset S$ is a minimal extension for each $x\in S\setminus R$. In particular, $R[x]$ is a co-atom for each $x\in S\setminus R$ (see Section 7). 
\end{definition}

We can remark that the above definition without ``$x\in S\setminus R$'' is  uninteresting  because ipso facto  this would mean that $R\subset S$ is minimal. 
The next proposition shows that  our definition of a co-pointwise minimal extension leads to a special case of pointwise minimal pairs. 

\begin{proposition}\label{7.5} Let $R\subset S$ be a ring extension. The following conditions are equivalent:
\begin{enumerate}
\item $R\subset S$ is a co-pointwise minimal extension;

\item  $R\subset S$ is a pointwise minimal pair such that $\ell[R,S]=2$;

\item   $R\subset S$ is a pointwise minimal pair and  the $R$-algebra  $S$   has a minimal system of generators whose cardinality is $2$. 
\end{enumerate}

In particular, a co-pointwise minimal extension has FCP.
\end{proposition}

\begin{proof}  
(1) $\Rightarrow$ (2) Assume that $R\subset S$ is  co-pointwise minimal. Let  $T\in[R,S]\setminus\{R,S\}$ and $x\in S\setminus T,\ y\in T\setminus R$. Consider the tower $R[y]\subseteq T\subset T[x]\subseteq S$. Since $R[y]\subset S$ is minimal, $R[y]=T$ and $T[x]=S$, so that $T\subset T[x]$ is minimal. If $T=R$, assume that $R\subset R[x]$ is not minimal, for some $x\not\in R$, so that there is $T'\in[R,R[x]]\setminus\{R,R[x]\}$. Let $y\in T'\setminus R$. Then $R[y]\subseteq T'\subset R[x]\subset S$ is absurd since $R[y]\subset S$ is minimal. Hence, $R\subset S$ is a pointwise minimal pair. Moreover, for any $x\in S\setminus R$, we get that $R\subset R[x]$ and $R[x]\subset S$ are minimal, giving $\ell[R,S]= 2$. In particular, a co-pointwise minimal extension has FCP.

(2) $\Rightarrow$ (3) Assume that $R\subset S$ is a pointwise minimal pair and $\ell[R,S]=2$. Then any maximal chain from $R$ to $S$ has length 2 and is of the form $R\subset T\subset S$, where $R\subset T$ and $T\subset S$ are minimal, whence monogenic. Pick $x,y\in S$ such that $T=R[x]$ and $S=T[y]$,  so that $S=R[x,y]$. Moreover, $S\neq R[x],R[y]$  and $R[x]\neq R[y]$,  since $R\subset S$ is not minimal. 

(3) $\Rightarrow$ (1) Assume now that $R\subset S$ is a pointwise minimal pair and $S=R[x,y]$ for some $x,y\in S$ such that $S\neq R[x],R[y]$ and $R[x]\neq R[y]$. Then $R\subset R[x]$  and $R[x]\subset R[x,y]=S$ are minimal. Let $z\in S\setminus R$. If $z\in R[x]$, then $R[z]=R[x]$ implies $R[z]\subset S$ is minimal. If $z\not\in R[x]$, then $R[x]\subset R[x,z]$ is minimal, giving $S=R[x,z]$, which implies that $R[z]\subset R[x,z]=S$ is minimal since $R\subset S$ is a pointwise minimal pair. Then, $R\subset S$ is  co-pointwise minimal.
\end{proof} 

In \cite{DS}Ê and \cite{D}, Dobbs and Shapiro studied  extensions of the form $R\subset T\subset S$, where $R\subset T$ and $T\subset S$ are minimal. Since a co-pointwise minimal extension has length 2, we may use their results. We will have more  details about the connection with these papers after  having characterized  co-pointwise minimal extensions in Section 5. 

Characterizations   for arbitrary  integral extensions  will  surprisingly lead  to  three special cases of the canonical decomposition.

\section{The case of an integral  extension over a field}

For an ideal $I$ of a ring and a positive integer $n$, we set $I^{[n]}:=\{x^n\mid x\in I\}$. 
In this section,  $k \subset S$ is an integral extension and $k$ is a field are  the riding hypotheses.  If $k\subset L$ is an algebraic field extension and $y\in L$, the minimal  polynomial of  $y$ over $k$ is denoted by $P_{k,y}(X)$.

\begin{proposition}\label{7.10} If   $k\subset S$ is  subintegral, then $S$ is a local ring with maximal ideal $N$.  Moreover, the following statements hold:
\begin{enumerate}
\item $k\subset S$ is  pointwise minimal  if and only if 
$N^{[2]}=0$.

\item  $k\subset S$ is a pointwise minimal pair if and only if $N^2=0$.
\end{enumerate}
\end{proposition}

\begin{proof} Since $k\subset S$ is subintegral, its  spectal map  is a homeomorphism  so that $S$ is a zero-dimensional local ring, with maximal ideal $N$. Then, $k\cong S/N$ gives  $S=k+N$  ($\star$).

(1) Assume that $k\subset S$ is  pointwise minimal  and let $x\in N\setminus\{0\}$ so that $x\not\in k$. Then, $k\subset k[x]$ is necessarily minimal ramified  because it is subintegral. Let $N'$ be the maximal ideal of $k[x]$, which implies that  $x\in N'=N\cap k[x]$. Moreover, $N'^2=0$, gives $x^2=0$. 

Conversely, assume that 
$N^{[2]}=0$.
 Then, $k\subset k[x]$ {\color{blue} is}  minimal ramified  when $x\in N\setminus\{0\}$  by Lemma \ref{1.3}. Let $y\in S\setminus k$, such that $y=a+x'$, for some $a\in k$ and $x'\in N\setminus\{0\}$, by ($\star$), whence  $k[y]=k[x']$ and $k\subset k[y]$ is  minimal. To conclude, $k\subset S$ is  pointwise minimal.

(2) Assume that $k\subset S$ is a pointwise minimal pair and let $x, y \in N$ be two different elements. Then $x^2=0=y^2$ by (1).  Since $k\subset k[x]$ is minimal ramified, $kx =N\cap k[x]$ is the maximal ideal of $k[x]$. If $y\in k[x]$, then $xy=0$, because $y\in kx$. If not, then $ k[x]\neq k[y]$ and $\{x, y\}$ is free over $k$. Moreover, $k[x]\subset k[x,y]$ is minimal ramified with conductor $kx$. Then, $xy\in kx$ gives that there is some $a\in k$ such that $xy=ax$. The same reasoning gives some $b\in k$ such that $xy=by$, so that $ax=by$ which implies $a=b=0=xy$. Then, $N^2= 0$.

Conversely, assume that $N^2=0$. Let $T\in[k,S]$ and $x\in S\setminus T$. Set $ N':=N\cap T$, which is the maximal ideal of $T$. As in (1), we can  assume that $x\in N$, so that $x^2=0\in N'$ and $T[x]=T+Tx$. Since $k\subset S$ is subintegral, so is $k\subset T$. In particular, $T=k+N'$. From $N^2=0$, we deduce $N'x\subseteq N^2=0\subseteq N'$, so that $T\subset T[x]$ is minimal ramified. 
\end{proof} 

\begin{lemma}\label{7.101} If 
$k\subset S$ is finite,  seminormal  and infra-integral,  there exists an integer $n$ such that $S\cong k^n$. 
\end{lemma}

\begin{proof} $k\subset S$ is \' etale by \cite[Lemma 5.6]{Pic 6} and then $S\cong k^n$ for some integer $n$ by \cite[Theorem 3.4]{Pic 6}.
\end{proof} 

\begin{proposition}\label{7.11} If 
 $k\subset S$ is seminormal and infra-integral, the following statements hold: 
\begin{enumerate}
\item Assume that $|k|\neq 2$. Then, $k\subset S$ is a pointwise minimal pair if and only if $k\subset S$ is a pointwise minimal extension and,  if and only if $k\subset S$ is a  minimal extension.
\item Assume that $|k|= 2$. Then  $k\subset S$ is always a pointwise minimal extension,  and  is a pointwise minimal pair  
 if and only if  $S\cong k^n$ with $n\leq 3$ and, if and only if $\ell[k,S]\leq 2$.
\end{enumerate}
\end{proposition}

\begin{proof} 
(1) We have the following implications: $k\subset S$ is   minimal  
$\Rightarrow k\subset S$ is a pointwise minimal pair $\Rightarrow k\subset S$ is  pointwise minimal.

Now, assume that $k\subset S$ is  pointwise minimal   and $|k|\neq 2$. By Lemma \ref{7.9}, $S$ is the union  of an upward directed family $\mathcal F$ of FCP extensions $R_{\alpha}$. For each $\alpha$, there exists an integer $n_{\alpha}\neq 0,1$ such that $R_{\alpha}\cong k^{n_{\alpha}}$, since $k\subset R_ {\alpha}$ is a finite seminormal infra-integral  extension by Lemma \ref{7.101}. We are going to show that $n_{\alpha}=2$. Deny. Let $e$ and $f$ be two elements of the standard basis of the $k$-vector space $ k^{n_ {\alpha}}$, so that $e^2=e,\ f^2=f$ and $ef=0$. It follows that $\{1,e,f\}$ is free because $n_\alpha > 2$. Let $\lambda\neq 0,1$ in $k$ and set $x=e+\lambda f$. Then,  $k\subset k[x]$ is minimal decomposed by Proposition \ref{7.90} since $x\not\in k$. It follows that $\{1,x\}$ is a basis of $k[x]$ over $k$ since $\dim_k(k[x])=2$.  Then, there exist $a,b\in k$ such that $x^2=a+bx$, giving $a+b(e+\lambda f)=e+\lambda^2f$, so that $\lambda^2=\lambda$, a contradiction since $\lambda\neq 0,1$. To conclude, we have $n_{\alpha}=2$ and then  $k\subset R_{\alpha}$ is minimal and  $k\subset S$ is minimal by Lemma  \ref{7.9}.
 
(2) Assume that $|k|= 2$ and let $x\in S\setminus k$. Since $k\subset S$ is seminormal infra-integral and $x$ is integral over $k$, we get that $k\subset k[x]$ has FCP  \cite[Theorem 4.2]{DPP2} and $k[x]\cong k^n$, for some integer $n$ by Lemma \ref{7.101}, from which we infer that $x=(x_i)_{i=1}^n$, where $x_i\in\{0,1\}$ for each $i$. Then, $x^2=x$ gives that $k\subset k[x]$ is minimal decomposed by Lemma \ref{1.3}, so that $ k\subset S$ is a pointwise minimal extension. 

Assume that $k\subset S$ is also a pointwise minimal pair. In view of Lemma  \ref{7.9}, $S$ is the union of an upward directed family $\mathcal F$ of FCP integral extensions $R_{\delta}$ such that $R_{\delta}\cong k^{n_{\delta}}$ for some integer $n_{\delta}$, for each $\delta$ (see (1)). Assume that there exists some $\delta$ such that $n_{\delta}\geq 4$, and let $e_{\alpha},\ e_{\beta}$ and $e_  {\gamma}$ be distinct elements of the standard basis of the $k$-vector space $R_ {\delta}$. Set $e:=e_{\alpha}+e_{\beta},\ f:=1+e$ and $x:=e_{\alpha}+e_{\gamma}$. Let $T:=k[e]=k+ke$, so that $x\not\in T$. Then, $k\subset T$ is minimal decomposed and $\mathrm{Max}(T)=\{ke,kf\}$ because $e^2=e$. 
Although $x^2=x$ and $T\subset T[x]$ is seminormal infra-integral, $T\subset T[x]$ is not minimal, because $ex=e_ {\alpha}\not\in T$ and $fx=e_{\gamma}\not\in T$, 
a contradiction since $k\subset S$ is  a pointwise minimal pair. Hence, $n_ {\delta}\leq 3$ for each $R_{\delta}\in\mathcal F$. Set $n:=\sup\{n_{\delta}\mid R_ {\delta}\in\mathcal F\}$. There is some $R_{\delta}\in\mathcal F$ with $n=n _{\delta}$. For $R_i\in\mathcal F$, there is $R_j\in\mathcal F$ such that $R_i,R _{\delta}\subseteq R_j$. Therefore,  $R_{\delta}=R_j$ and $R_i\subseteq R_ {\delta}$ for each $R_i\in\mathcal F$, giving $S=R_{\delta}\cong k^n$, with $n\leq 3$. 

Conversely, assume that $S\cong k^n$, with $n\leq 3$. If $n=2$, then $S\cong k^2$ and $k\subset S$ is minimal and so a pointwise minimal pair. If $n=3$, then $S\cong k^3$ and $\ell[k,S]=2$, so that $k\subset S$ is a pointwise minimal pair. The last equivalence is then obvious.
 \end{proof}

\begin{proposition}\label{7.13} \cite[Proposition 4.16]{CDL} If  
 $k\subset S$ is t-closed, then $S$ is a field and the following statements are equivalent:

\begin{enumerate}

\item  $k\subset S$ is a pointwise minimal extension;

\item $k\subset S$ is a pointwise minimal pair; 

\item    $k\subset S$ is either a minimal separable field extension, or   a height one radicial extension. 
\end{enumerate}
\end{proposition}

\begin{proof} In view of \cite[Lemme 3.10]{Pic 1}, $S$ is a field. Then,    \cite[Proposition 4.16]{CDL}  shows that (1) $\Leftrightarrow$ (3). A pointwise minimal pair is  pointwise minimal. To end, assume that (3) holds. If $k\subset S$ is minimal, we are done. Assume that $k\subset S$ is a  height one radicial  
  extension where $\mathrm{c}(k) =p$  and  $x^p\in k$ for each $x\in S$.
 Let $T\in[k,S]$ and $x\in S\setminus T$. Then, $T$ is a field and $x^p\in k\subseteq T$ implies that $[T[x]:T]=p$. Hence,  $T\subset T[x]$ is  minimal  and $k\subset S$ is a pointwise minimal pair.
\end{proof}

  We now  intend to characterize pointwise minimal extensions $k\subset S$, where $k$ is a field, that do not satisfy the conditions of Propositions \ref{7.10}, \ref{7.11} and \ref{7.13}. The following lemma gives a necessary condition on the seminormalization and the t-closure. 

\begin{lemma}\label{7.14} If  $k\subset S$ is  pointwise minimal  such that ${}_S^tk\neq{}_S^+k$, then, $k\subset S$ is  seminormal  and infra-integral.

\end{lemma}

\begin{proof} In order to show that $k\subset S$ is seminormal infra-integral, it is enough to show that $k={}_S^+k$ and $S={}_S^tk$ (Definition \ref{7.90.0}). So, deny that these two equations hold.

Assume first that $k\neq{}_S^+k$. Then, there exist $x\in{}_S^+k\setminus k$ and $y\in{}_S^tk\setminus{}_S^+k$. We now observe that $k\subset k[x]$ is  minimal ramified  and $k\subset k[y]$  is minimal decomposed. 
 Indeed, $x\in{}_S^+k$ shows that $k\subset k[x]$ is minimal ramified by Proposition \ref{7.90}(1). Moreover, $y\not\in{}_S^+k$ shows that $k\subset k[y]$  is not ramified and  
 then is decomposed since $y\in{}_S^tk$, by Proposition \ref{7.90}(2). Set $z:=x+y$. Then,  $z\not\in{}_S^+ k$ since $y\not\in{}_S^+k$,  so that $z\not\in k[x]$, but $z\in{}_S^tk$ because $x,y\in{}_S^tk$. This implies that $k\subset k[z]$ is  minimal  and infra-integral, necessarily decomposed, because not ramified. Moreover, $k[z,y]=k[x,y]=
  k[x]k[y]=k[y]k[z]$. We  first get that $k\subset k[x,y]$ is  the composite of a minimal ramified extension $k\subset k[x]$ and a minimal decomposed extension $k\subset k[y]$. It follows from  \cite[Proposition 7.6]{DPPS}, that $k[y]\subset k[x,y]$ is either ramified, or a tower of two ramified extensions. Now, 
 since
  $k\subset k[z,y]=k [x,y]$ is the composite of two minimal decomposed extensions $k\subset k[y]$ and $k\subset k[z]$, from \cite[Proposition 7.6]{DPPS}, we deduce  that $k[y]\subset k[x,y]=k[z,y]$ is either decomposed, or a tower of two decomposed extensions, a contradiction.
Then $k={}_S^+k$ and $k\subset S$ is seminormal.

Assume now that $S\neq{}_S^tk$. Then, there exist $x\in{}_S^tk\setminus k$ and $y\in S\setminus{}_S^tk$. We intend to prove that $k\subset k[x]$ is  minimal decomposed  and $k\subset k[y]$ is  minimal inert. 
To begin with,  $x\in{}_S^tk$ shows that $k\subset k[x]$ is  minimal,   seminormal and infra-integral, and then    decomposed   by Proposition \ref{7.90}(3). Moreover, $y\not\in{}_S^tk$ shows that $k\subset k[y]$ is not infra-integral and then cannot be decomposed (Proposition \ref{7.90}(3)). 
 Then, $k\subset k[y]$ is inert.
There is no harm to choose $x$ such that $x^2=x\ (*)$ 
 because $k\subset k[x]$ is decomposed (Lemma \ref{1.3}). The maximal ideals of $k[x]$ are $M_1:=kx$ and $M_2:=k(1-x)$ and there exist $M'_1,M'_2\in\mathrm{Max}(k[x,y])$ such that $M'_i$ lies over $M_i$ for $i=1,2$. In particular, $x\in M'_1\setminus M'_2$ because $1-x\in M_2 \subseteq M'_2$. Now, $k\subset k[y]$ is a field extension  because $k\subset k[y]$ is inert, giving that $y$ is a unit in $k[y]$, and also in $k[x,y]$. Set $z:=xy\in k[x,y]$. Then, $z\in M'_1\setminus M'_2\ (**)$ 
 because $x\in M_1'$ and if $z\in M'_2$, then, $y\in M'_2$ since $x\not\in M'_2$, a contradiction. In particular, $z\not\in k$, since $z\in k$ would imply $z\in M'_1\cap k=0$. Moreover, $k\subset k[z]$ is  minimal and seminormal. We claim that $k\subset k[z]$ is not inert. Deny. Then, $k [z]$ is a field, so that $z$ is a unit in $k[z]$, and also in $k[x,y]$, implying that $x$ is also a unit in $k[x,y]$, a contradiction with $(*)$. It follows that $k\subset k[z]$ is decomposed. Hence, $k[z]$ has two maximal ideals $N_1$ and $N_2$. Using $(**)$  and setting $N_i:=M'_i\cap k[z]$ we get that $z\in N_1 \setminus N_2$. Indeed, $z\in M'_1\cap k[z]$ and $z\not\in M'_2\cap k[z]$. In particular, there exists $t\in N_1$ such that $N_1=kt,\ N_2=k(1-t)$, with $t^2=t\ (***)$  since $k\subset k[z]$ is decomposed. Then, $z=at$, for some $a\in k\setminus\{0\}$. Set $b:=a^{-1}\neq 0$ so that $t=bz=bxy$. By $(***)$, we have $b^2x^2y^2=bxy=b^2xy^2$,  by $(*)$, giving $x=bxy$, that is $x(1-by)=0 \in M'_2$ in $k[x,y]$. This implies that $1-by\in M'_2\cap k[y]=\{0\}$ because $x\not\in M'_2$ and $k[y]$ is a field.  Hence, $y\in k$, a contradiction. Then, $S={}_S^tk$ and $k\subset S$ is   infra-integral.
 \end{proof}

The last case to consider is an extension of the form $k\subset T\subset S$, where 
 $k\subset T$ is subintegral and $ T\subset S$ is t-closed. 
 
 \begin{proposition}\label{7.15}  If $k\subset S$ is such that $k\neq{}_S^tk={}_S^+k\neq S$, then $T:={}_S^tk$ and $S$ are both local rings sharing a same  maximal ideal $N$. Moreover, $k\subset S$ is a pointwise minimal extension if and only if the following conditions are satisfied:  
\begin{enumerate}

\item $N^{[2]}=0$.

\item  
 $k\subset S$ is radicial of height one, whence also $k\subset S/N$.
\end{enumerate}
 
 Such an extension is never a pointwise minimal pair.
 \end{proposition}

 \begin{proof}  
  We get that $(k:S)=0$ since $k$ is a field. By Lemma \ref{7.90.1}, it follows that $N:=(T:S)$ is the only maximal ideal of $T$ and $S$.  
Assume that $k\subset S$ is  pointwise minimal. Then so is  $k\subset T$ and (1) holds, since it satisfies the conditions of Proposition \ref{7.10}. 

We now show (2). Fix some    $x\in N\setminus k$ , so that $x^2= 0$ by (1). 
Then, $k\subset k[x]$ is minimal ramified  by Proposition \ref{7.90}(1) because $x\in {}_S^+k$. 
 Let $y\in S\setminus T$. It follows that $k\subset k[y]$ is a minimal extension, which is  neither ramified nor decomposed, since $y\not\in {}_S^tk$. So,   
$k\subset k[y]$ is  minimal inert 
 and then t-closed,  whence $k[y]$ is a field by Proposition  \ref{7.13}. 
   Now set $z:=x+y$. We get that $z\not\in {}_S^tk$ since $x\in{}_S^tk$ and $y\not\in{}_S^tk$. Then,    $k\subset k[z]$ is  minimal 
   inert 
 by the same proof used for $k\subset k[y]$
   and then  a field extension  by Proposition \ref{7.13}. Let $P(X) = P_{k,z}(X)\in k[X]$,  with derivative $P'(X)$. From $P(x+y)=0$,  we deduce that $P(y)+xP'(y)=0$. Since $\{1,x\}$ is a  basis of $k[x,y]$ over the field $k[y]$, we get $P(y)=P'(y)=0$,  from which we infer that $P'(X)=0$, because $P(X)$ is irreducible and $P(y)=0$ shows that $P(X)= P_{k,y}(X)$. In particular, $P(X)$ is not  separable.  Hence, $k\subset k[y] $  is   radicial  by Lemma ~\ref{1.3.1} and $\mathrm{c}(k)= p\in \mathbb{P}$. Now, $k\cong T/N\subset S/N$ is a field extension. If $\bar y$  is the class of $y$ in $S/N$,  $P(X)= P_{k,\bar y}(X)$ and then  $\bar y$ is a  radicial element. 
 Now from $T=k+N$  ($k\subset T$ is subintegral), we deduce that  each
 $t\in T$ is of the form  $t=b+m$, for some $b\in  k$ and $m\in N$. Then, $t^p=b^p+m^p=b^p\in k$ by (1).   
 If $y$ belongs to $S\setminus T$, the above proof shows that the minimal extension $k\subset k[y]$ is radicial  and $y^p\in k$ by Lemma~\ref{1.3.1}.  Then, $k\subset S$ is radicial of height one, whence also  $k\subset S/N$.
The proof of (2) is now complete. 

Conversely, assume that (1) and (2) hold. Proposition  \ref{7.10}(1) entails  that $k\subset k[x]$ is minimal for each $x\in T\setminus k$. Let $y\in S\setminus T$ and $\bar y$ its class  in $S/N$. By (2), we get that $y^p=a\in k$, giving $\bar y^p=a\in k$. From  (2)  we deduce  that $k\subset k[\bar y]$ is radicial  and necessarily $ P_{k,\bar y} =X^p-a$. It follows that $X^p -a= P_{k,y}(X)$ and then  $k[y]\cong k[X]/(X^p-a)$. Therefore, $k[y]$ is a field such that $[k[y]:k]=p$ and  then $k\subset k[y]$ is a minimal field extension. The proof is complete.
 
We claim that under these conditions, $k\subset S$ is not a pointwise minimal pair. Take again $x\in T\setminus k$ and $y\in S\setminus T$. Then $k\subset k[x]$ is minimal ramified, and $k\subset k[y]$ is minimal inert, but   $k[x]\subset k[x,y]$ is not minimal  \cite[Proposition 7.4]{DPPS}.
\end{proof}

\begin{remark}\label{7.151} We exhibit  an extension $k\subset S$ which satisfies the hypothesis of Proposition  \ref{7.15}, such that 
$N^{[2]}=0$
 and such that $k\subset S/N$ is radicial of height one, but which is not  pointwise minimal.

Let $k$ be a field with  $\mathrm{c}(k)= 2$ 
  and such that $k\neq k^{[2]}$.
  Set $R:=k[T]/(T^2)=k[t]=k+kt$, where $t$, the class of $T$, satisfies $t^2=0$. Then, $k\subset R$ is  minimal ramified, so that $R$ is a zero-dimensional local ring  with maximal ideal $M:=kt$. 
 Pick some
 $a\in k\setminus k^{[2]}$ and set $S:=R[X]/(X^2-a-t)=R[x]=R+Rx
=k+kt+kx+ktx$, 
where $x$, the class of $X$ satisfies $x^2=a+t$. Then, $k\subset S$ is an integral extension. 
 Set $R':=R[tx]=R+Rtx=k+kt+ktx$. Then $R'$ is a zero-dimensional local ring with maximal ideal $N:=kt+ktx$, satisfying $N^{[2]}=0$ and $R\subset R'$ is a minimal ramified extension such that $M=(R:R')$ and $R'/N\cong k$. It is easy to check that $N$ is an ideal of $S$, such that $S/N\cong k+k\overline x$, where $\overline x$ is the class of $x$ in $S/N$,  and $N=(R':S)$. Moreover, $\overline x$ satisfies $\overline x^2=\overline a=a\in k$, so that $S/N\cong k[\overline x]\cong k[Y]/(Y^2-a)$. Therefore,  $N\in\mathrm{Max}(S)$ and $k\subset S/N$  is minimal  radicial of height one. Then, $R'\subset S$ is minimal inert and $(S,N)$ is local. We infer that $R'={}_S^+k={}_S^tk$. But, $k\subset S$ is not pointwise minimal. Indeed, $t\in k[x]$, so that $S=k[x,t]=k[x]$ is such that $k\subset k[x]$ is not minimal. 
\end{remark}

\section{Arbitrary integral extension}

Gathering results of Propositions \ref{7.10}, \ref{7.11}, \ref{7.13} and \ref{7.15}, we are now able to characterize pointwise minimal extensions and pointwise minimal pairs. We 
 first 
give a statement for an integral extension $k\subset S$, where $k$ is a field, and then for an arbitrary  integral extension. 
 As we saw in the previous sections, some pointwise minimal extensions are minimal. We first  get rid of these  cases  in the next result.

\begin{proposition}\label{7.160} Let $R\subset S$ be an $M$-crucial  extension, satisfying one of the following  mutually exclusive conditions:

\begin{enumerate} 
\item $R\subset S$ is integrally closed.

\item $R\subset S$ is seminormal infra-integral such that $|R/M|\neq 2$.

\item $R/M\subset S/M$ is a separable field extension.
\end{enumerate} 

Then, $R\subset S$ is pointwise minimal if and only if $R\subset S$ is  minimal.
\end{proposition}
\begin{proof} Use Propositions \ref{7.7} for (1), \ref{7.11} (1) for (2) and \ref{7.13} for (3).
\end{proof}}

In the next result, we use Lemma~\ref{7.90.1}  in conditions (1) and (2), because  $k\subset S$ is of the form $k\subseteq T \subseteq S$, where $k\subseteq T$ is subintegral and $T\subseteq S$ is $t$-closed.  

\begin{theorem}\label{7.16} Let $k\subset S$ be a non-minimal integral extension,  where $k$ is a field, and $N: = \sqrt[S]{0}$. Consider the following   conditions:
 
\begin{enumerate} 
\item 
 ${}_S^tk={}_S^+k $,   $N^{[2]}=0$  
 and  if ${}_S^tk\subset S$, then $k\subset S$ is  a  height one radicial  extension.
 \item $k\subset S$ is subintegral and     $N^2=0$.
 
\item $|k|= 2$ and $k\subset S$ is a seminormal infra-integral extension. 

\item $|k|= 2,\ k\subset S$ is a seminormal infra-integral extension and $S\cong k^n$ with $n\leq 3$ (equivalently $\ell[k,S]\leq 2$).
\end{enumerate}
 Then, $k\subset S$ is  pointwise minimal  if and only if one of the mutually exclusive conditions (1) or (3) holds and $k\subset S$ is a pointwise minimal pair if and only if one of the  mutually exclusive conditions (1) with ${}_S^tk=k$,  
(2) or (4) holds.
\end{theorem}

\begin{proof} Assume that one of conditions (1) or  (3) holds (resp.; (1)) with ${}_S^tk=k$, 
  (2) or (4)), then, $k\subset S$ is a pointwise minimal extension (resp. pair) in view of Propositions  \ref{7.10}, \ref{7.11}, \ref{7.13} and \ref{7.15}, 
    excluding the minimal cases.

Conversely, assume that $k\subset S$ is pointwise minimal. 
If ${}_S^tk\neq{}_S^+k$, then $k\subset S$ is seminormal infra-integral by Lemma  \ref{7.14} and satisfies    
 (3) by Proposition \ref{7.11}. In particular, if $k\subset S$ is a pointwise minimal pair, Proposition \ref{7.11} shows also  that $k\subset S$ satisfies  
 (4).
 
  Assume now that ${}_S^tk={}_S^+k$. If $k\neq{}_S^tk={}_S^+k\neq S$, then $k\subset S$ satisfies (1) by Proposition \ref{7.15}. Moreover,  Proposition \ref{7.15} says that $k\subset S$ is not a pointwise minimal pair. 
  
  Two  cases are remaining.  The first one is when  
 $k={}_S^tk={}_S^+k\neq S$, that is $k\subset S$ is t-closed, and Proposition \ref{7.13} gives (1) for both  a pointwise minimal extension and a pointwise minimal pair. The second  case is 
 when $k\neq{}_S^tk={}_S^+k= S$, that is $k\subset S$ is subintegral, and Proposition \ref{7.10} gives (1) for a pointwise minimal extension and (2) for a pointwise minimal pair.
\end{proof}

 Using  the four conditions of Theorem \ref{7.16}, we are now able to give  a complete characterization of pointwise minimal extensions and  pairs.

 \begin{theorem}\label{7.17} Let $R\subset S$ be a non-minimal integral 
extension with $M:=(R:S)\in \mathrm{Max}(R)$. Consider the following   conditions:
 
\begin{enumerate} 
\item  ${}_S^tR={}_S^+R$, 
 $\sqrt[S]{M}^{[2]}\subseteq M$ 
 and, if ${}_S^tR\subset S$, then $R\subset S$ is  a  height one radicial  extension.

\item $R\subset S$ is subintegral and   $\sqrt[S]{M}^2\subseteq M$.
\item $|R/M|= 2$ and $R\subset S$ is  seminormal  and infra-integral. 

\item $|R/M|= 2,\ R\subset S$ is  seminormal  and infra-integral  and $S/M\cong (R/M)^n$ with $n\leq 3$ (equivalently $\ell[R,S]\leq 2$)..
\end{enumerate}
Then, $R\subset S$ is  pointwise minimal  if and only if  one of the mutually exclusive conditions (1), (3) holds and is a pointwise minimal pair if and only if  one of the mutually exclusive conditions (1) 
 with ${}_S^tR=R$,
 (2),  (4) holds.
\end{theorem}

\begin{proof}   We can reduce to the case where  $R$ is a field by using $R/(R:S) \subset  S/(R:S)$, Proposition \ref{7.2}(2) and  Theorem \ref{7.16}.

 First, we can remark that ${}_{S/M}^t(R/M)=({}_S^tR)/M$ and ${}_{S/M}^+(R/M)=({}_S^+R)/M$ \cite[Proposition 2.10]{Pic 1} for the t-closure. For the seminormalization, a straightforward  proof shows that $R/M\subseteq ({}_S^+R)/M$ is subintegral and $({}_S^+R)/M\subseteq S/M$ is seminormal. It follows that the properties of subintegrality, infra-integrality, seminormality and t-closedness are the same for the extensions $R\subset S$ and $R/M\subset S/M$. A similar result  holds for the minimality of $R\subset S$ and $R/M\subset S/M$,
 the cases we have to exclude.
 \end{proof}

\begin{corollary}  Let $R\subset S$ be a  non-minimal pointwise minimal integral extension and  $M=(R:S)$. The following cases occur:

$\mathrm{(a)}$  $R\subset S$ is subintegral;

 $\mathrm{(b)}$  $R\subset S$ is infra-integral and seminormal;

 $\mathrm{(c)}$ $R\subset S$ is $t$-closed and radicial of height one;
  
 $\mathrm{(d)}$ $R\subset S$ is radicial of height one and of the form $R \subset T \subset S$, where $R\subset T$ is infra-integral and $T\subset S$ is $t$-closed.
  In this case, $T\subset S$ is  pointwise minimal.

For each $N\in \mathrm{Max}(S)$ above $M$, if (a) or (b) holds     $R/M \subseteq S/N$ is an isomorphism and    pointwise minimal   if  (c) or (d) holds. 
 \end{corollary}

\begin{proof}   Following  the different cases of Theorem \ref{7.17}, we get   (1) $\Rightarrow $ 
 either (a) or  (c) or (d), and  
(3) $\Rightarrow $ (b). 

We now prove the last part of the Corollary.
Since   $R\subset S$ is infra-integral if (a) or (b) holds,  $R/M \subseteq S/N$ is an isomorphism.

 Suppose that (c) holds.  As we already observed in the proof of 
 Theorem \ref{7.17}, 
  $R/M\subset S/M$ is  t-closed, so that $M$ is  the only maximal ideal of $S$ lying over $M$ 
 (Lemma \ref{7.90.1}). 
 Then, $R/M\subset S/M$ is  pointwise minimal by Proposition  \ref{7.2} (2).

  Suppose that (d) holds and set $
  T:={}_S^tR={}_S^+R$. 
By  Lemma \ref{7.90.1}, $N:= \sqrt[S]{M}$ 
is  the unique maximal ideal of $S$ lying over $M$,  
so that $N=(T:S)$ and  also the only maximal ideal of $T$ lying over $M$. By subintegrality of $R\subset T$,   we get  $R/M\cong T/N$. For $x\in S$, let $\overline x$ be the class of $x$ in $S/N$ and $\tilde x$  the class of $x$ in $S/M$. In view of Proposition  \ref{7.2}, $R/M\subset S/M$ is  pointwise minimal, with $T/M\neq R/M,S/M$, so that the extension $R/M\subset S/M$ satisfies the case (1) of Theorem \ref{7.16}. Then,  
${\tilde x}^p\in R/M$ for each $\tilde x\in S/M$, where $p:=\mathrm{c}(R/M)$. 
Therefore,  there are $a\in R$ and $m\in M\subseteq N$ such that $x^p=a+m$, giving $({\overline x})^p=\overline {x^p}=\overline a\in R/M\cong T/N$ in $S/N$. Then, $R/M\cong T/N\subset S/N$ is a height one radicial extension, so that $R/M\subset S/N$ is pointwise minimal by Theorem \ref{7.17} (1). To end, under this condition,  $T/N\subset S/N$ is pointwise minimal, and so is $T\subset S$.
\end{proof}

If the  pointwise minimal integral extension $R\subset S$  has FCP, we can improve Corollary \ref{7.30} by using Theorem \ref{7.17}.

 \begin{corollary}\label{7.169} Let $R\subset S$ be an FCP pointwise minimal integral extension with conductor $M\in \mathrm{Max}(R)$ and set $k:= R/M$ and  $p:= \mathrm c(k)$.  Only the following three case can  occur: 
 \begin{enumerate}
\item $R\subset S$ is infra-integral. Then, $\dim_{k}(S/M)=1+\ell[R,S]$.

\item $R\subset S$ is t-closed. Then, either $\ell[R,S]=1$ when $R\subset S$ is minimal, or $ p\in \mathbb{P}$  and $\dim_{k}(S/M)=p^{\ell[R,S]}$.

\item $R\neq{}_S^+R={}_S^tR\neq S$ and $ p\in \mathbb{P}$.  Then,  $\dim_{k}(\sqrt[S]{M}/M)=\ell[R,{}_S^tR]$ and  $\dim_{k}(S/M)=\dim_{k}(\sqrt[S]{M}/M)+p^{\ell[R,S]-\dim_{k}(\sqrt[S]{M}/M)}$.
\end{enumerate}
\end{corollary}
\begin{proof}  A quick look at Theorem \ref{7.17} shows that each of its  statements   meets one of the conditions of the corollary. We use the characterization of a pointwise minimal extension of Theorem \ref{7.16}. Setting $S':=S/M$, we get that $\ell[R,S]=\ell[k,S']$. Moreover, the extensions $R\subset S$ and $k\subset S'$  have the same properties with respect to the canonical decomposition, and there is a bijection $[R,S]\to [k,S']$ given by $U\mapsto U/M$. 

(1) Assume that  $k\subset S'$ is subintegral. Then, $S'$ is a zero-dimensional local ring with maximal ideal  $N'$. 
It follows 
that  $\ell[k,S']=\dim_k(N')$  \cite[Lemma 5.4]{DPP2} and  $S'=k\bigoplus N'$ combine to yield  $\dim_k(S')=1+\ell[k,S']$, so that $\dim_{R/M}(S/M)=1+\ell[R,S]$.

Assume that $k\subset S'$ is seminormal infra-integral. 
 If $k\subset S'$ is minimal, then $S'\cong k^2$ and $\dim_k(S')=2=1+\ell[k,S']$. If $k\subset S'$ is not minimal, by Theorem  \ref{7.16} (3)
$S'\cong k^n$ for some integer $n$, giving that $\dim_k(S')=n$, with $\ell[k,S']=n-1$. We still get $\dim_k(S')=1+\ell[k,S']$, so that $\dim_{R/M}(S/M)=1+\ell[R,S]$.

(2) Assume that $k\subset S'$ is t-closed. Then, either $k\subset S'$ is a minimal separable field extension, giving $\ell[k,S']=\ell[R,S]=1$ or $k\subset S'$ is  a height one radicial extension with $\mathrm{c}(k)=:p.$
 Hence, there exists an integer $m$ such that $[S':k]=p^m$. It follows that $\ell[k,S']=m$ and $\dim_k(S')=p^m$, so that $\dim_k(S')=p^{\ell[k,S']}$ and $\dim_{R/M}(S/M)=p^{\ell [R,S]}$.

(3) In the last case, set $T':={}_{S'}^+k={}_{S'}^tk$. Let $N'$ be the maximal ideal of $T'$. Since $T'\subset S'$ is t-closed, $N'$ is also the maximal ideal of $S'$ by Proposition \ref{7.15}, and then $N'=(T':S')$. We may use case (1) for the extension $k\subset T'$, so that $\dim_k(T')=1+\ell[k,T']=1+\dim_k(N')$. Moreover, $k\cong T'/N'\subset S'/N'$ is 
 a height one radicial extension with $\mathrm{c}(k)=:p.$
 Then, there exists an integer $m$ such that $[S'/N':k]=p^m$. It follows that $\ell[T',S'] =\ell[T'/N',S'/N']=m$ and $\dim_k(S'/N')=\dim_k(S')-\dim_k(N')=p^m$. But $\ell[k,S']=\ell[k,T']+\ell[T',S']\ (*)$ by \cite[Lemma 1.5]{Pic 4}. Now, $T'=k \bigoplus N'$ gives that $\dim_k(N')=\dim_k(T')-1=\ell[k,T']$. By $(*)$, we get $\ell[k,S']=\dim_k(N')+m$, so that $m=\ell[k,S']-\dim_k(N')$. To end, $\dim_k(S')=\dim_k(N')+p^m=\dim_k(N')+p^{\ell[k,S']-\dim_k(N')}$, giving $\dim_{R/M}(S/M)=\dim_{R/M}(N/M)+p^{\ell[R,S]-\dim_{R/M}(N/M)}$. 
 In particular, $\dim_{R/M}(N/M)=\ell[R,{}_S^tR]$.
\end{proof}

Theorem \ref{7.17} allows us to characterize co-pointwise minimal extensions via Proposition \ref{7.5}.

\begin{corollary}\label{7.18} An extension $R\subset S$ is  co-pointwise minimal  if and only if $M:= (R:S)\in \mathrm{Max}(R)$ and one of the  conditions below  holds:

\begin{enumerate}
\item $R\subset S$ is subintegral and 
 for
 $N:= \sqrt[S]{M}$, then $\dim_{R/M}(N/M)=2$ and $N^2\subseteq M$.

\item $|R/M|=2$ and $S/M\cong (R/M)^3$.

\item $R/M\subset S/M$ is a  height one radicial  
field extension of degree $p^2$, where $p:=\mathrm{c}(R/M)$.  
\end{enumerate}
\end{corollary}

\begin{proof} Proposition \ref{7.5} tells us  that  $R\subset S$ is  co-pointwise minimal  if and only if it is a pointwise minimal pair such that $\ell[R,S]=2$.  In particular, under these conditions, $R\subset S$ has FCP. 
Then, it is enough to consider the following conditions of Theorem \ref{7.17}: (2),  (4) and (1) with ${}_S^tR=R$, and then use  
 Corollary \ref{7.169}. Indeed, $R\subset S$ needs  to be  integral  because  of Proposition \ref{7.7} and Corollary \ref{7.8}. 
If $R\subset S$ is subintegral and a pointwise minimal  pair, we have $\ell[R,S]=\dim_{R/M}(N/M)$,
(Condition (1) of Corollary \ref{7.169}). 
Then, if $R\subset S$ is  co-pointwise minimal and   subintegral,   $\ell[R,S]=\dim_{R/M}(N/M)=2$, because $R \subset S$ is pointwise minimal  and $N^2\subseteq M$. Conversely, if $R\subset S$ is subintegral with $(R:S)=M$ and if 
$N:= \sqrt[S]{M}$
 is such that $\dim_{R/M}(N/M)=2$  and $N^2\subseteq M$, there exist $x,y\in N\setminus M$ such that $N=M+Rx+Ry$, and $S=R+Rx+Ry$, with $x,y\not\in R$, because $S/N\cong R/M$. Then, 
 $R\subset S$ is a pointwise minimal pair by Theorem \ref{7.17} (2) and 
$\dim_{R/M}(S/M)
\leq 3$, 
so that $\ell[R/M,S/M]\leq 2$. Since it cannot be minimal (if not,   $\dim_{R/M}(N/M)=1$), then, $\ell[R/M,S/M]=2$, which implies $\ell[R,S]=2$ and $R\subset S$ is co-pointwise minimal. 
 Assume that  $R\subset S$ is seminormal infra-integral. Then, $R\subset S$ is co-pointwise minimal  if and only if it is a pointwise minimal pair such that $\ell[R,S]=2$ if and only if 
$|R/M|=2$ and $S/M\cong(R/M)^3$  by Theorem  \ref{7.17} (4).
 At last, under condition (1)
 of Theorem \ref{7.17} 
with ${}_S^tR=R$, then
 $\mathrm c(R/M) = p\in \mathbb P$ and $\ell[R,S]= 2 \Leftrightarrow[S/M:R/M]=p^2$, 
 because $\ell[R,S]=\ell[R/M,S/M]$. 
\end{proof}

\section{Applications and examples}

\subsection{Some theoretical results}
  
In view of Theorem \ref{7.17}, we can complete  the results of  \cite[Theorem 4.9 and Proposition 4.11]{CDL}. 

\begin{proposition}\label{7.170} Let $(R,M)$ be  a local ring,  $R\subset S$  an integral  extension with conductor $M$  and let $J$ be the Jacobson radical of $S$.

\begin{enumerate}
\item  If  $R\subset S$ is a pointwise minimal extension, then,  
$J^{[2]}\subseteq  M$.
\item Conversely, if $J$ is an ideal of $S$,  such that $J\nsubseteq R$ and $J^{[2]}\subseteq M$, 
then $R\subset R+J$ is a pointwise minimal extension 
 which satisfies  condition  
(1)  of Theorem \ref{7.17}.
 \end{enumerate}
\end{proposition}

\begin{proof} (1) Since $M=(R:S)$ is the maximal ideal of $R$, $M$ is contained in any maximal ideal of $S$, so that $M\subseteq J$ and $J/M$ is the Jacobson radical of $S/M$. 
 If  $R\subset S$ is a non-minimal  pointwise minimal extension, $R\subset S$ satisfies one of the conditions (1) or (3)  of Theorem \ref{7.17}. 
In case (3), $R\subset S$ is a seminormal extension with conductor $M$, which is a radical ideal in $S$, and actually  
  the intersection of the
 maximal ideals of $S$, since $M$ is maximal in $R$, 
whence, $J=M$. In case (1), $S$ is a local ring. It follows that its maximal ideal is $J$ giving $x^2\in M$ for each $x\in J$. 
 The same holds if $R\subset S$ is minimal.

(2) Conversely, let $J$ be an ideal of $S$ such that  $J\nsubseteq R$ and 
$J^{[2]}\subseteq M$. Set $T:=R+J$ and let $z\in T\setminus R$. There exist some $a\in R,\ y\in J$ such that $z=a+y$. Moreover $R[z]=R[y]$ and $M=(R:R[y])$. Now, $y^2\in M$ and $My\subseteq M$ gives that $R\subset R[y]=R[z]$ is minimal ramified
 by Lemma \ref{1.3} 
(consider the extension $R/M\subset R[y]/M$). Then, $R\subset T$ is  pointwise minimal.  
 We are going to show that $T$ is a local ring. Deny. There exist two maximal ideals $M_1$ and $M_2$ of  $T$ satisfying $M_i\cap R=M$ since $R\subset T$ is integral. Since $M_1+M_2=T$, there exist $x_i\in M_i\setminus M_j$, for $\{i,j\}=\{1,2\},\ i\neq j$ such that $x_1+x_2=1\ (*)$. In particular, $x_i\not\in R$. But $R\subset R[x_i]$ is ramified for each $i$ and $R[x_i]$ has a unique maximal ideal $M'_i=M_i\cap R[x_i]$. Moreover, $x_i\in M'_i$ leads to $x_i^2\in M$. Then, $(*)$ gives $(x_1+x_2)^2=x_1^2+x_2^2+2x_1x_2=1$, so that $2x_1x_2\in R\cap M_i=M$, which implies $1\in M$, a contradiction. Then, $T$ is a local ring with a maximal ideal $N$ such that $N^{[2]}\subseteq M$. Indeed, for each $x\in N$, the local ring $R[x]$ has $N\cap R[x]$ for unique maximal ideal and $(N\cap R[x])^2\subseteq M$ since  $R\subset R[x]$ is minimal ramified. Moreover, since $R+J\subseteq R+N\subseteq T=R+J$, we get that $T=R+N$, so that $(R+N)/N\cong R/(R\cap N)\cong R/M$, which shows that  $R\subset T$ is subintegral, and satisfies condition (1) of Theorem \ref{7.17}.
 \end{proof}
 
 An extension $R\subseteq S$ is termed a {\it quadratic extension} if $S_t:= R+Rt\in [R,S]$ (whence $S_t=R[t]$), for each $t\in S$.

\begin{proposition}\label{7.12} A quadratic seminormal infra-integral FCP  
extension $R\subset S$, where $(R,M, k:= R/M)$ is a local ring, is a pointwise minimal extension, and  is minimal
 if and only if $|k|>2$.
\end{proposition}

\begin{proof} Since $R\subset S$ is quadratic,  $S_t =R[t]$ for each $t\in S\setminus R$. Moreover, $C_t:=(R:S_t)$ is a radical ideal of $S_t$ and $R$, and $R/C_t$ is Artinian by \cite[Lemma 4.8, Theorem 4.2]{DPP2}, so that $C_t=M$. From $\dim_{k}(S_t/M)\leq 2$, we deduce that $k\subset S_t/M$ is minimal, and so is $R\subset S_t$. Hence $R\subset S$ is a pointwise minimal extension. 

 If $|k|>2$, then, $k\subset S/M$, and $R\subset S$ are minimal extensions by Proposition \ref{7.11} (1). If $|k|=2$, the proof of Proposition \ref{7.11} (2) shows that $k\subseteq k^n$ is quadratic for any integer $n>1$.  
\end{proof}

An integral extension $R\subset S$ has FIP as soon as $M:=(R:S)\in\mathrm{Max}(R)$ is such that $R/M\subset S/M$ satisfies condition (4) of Theorem \ref{7.17}. The next proposition shows that in many cases, a pointwise minimal  FIP integral extension is actually a minimal extension, 
 completing Proposition \ref{7.160}.

\begin{proposition}\label{7.171} Let $R\subset S$ be an integral FIP extension with conductor $M$. Assume that either $|R/M|=\infty$ or $R\subset S$ is t-closed. Then,  $R\subset S$ is a pointwise minimal extension if and only if it is minimal. \end{proposition}

\begin{proof} One implication is obvious. 
Assume that $R\subset S$ is  pointwise minimal. In view of Proposition  \ref{7.2}(1), $M$ is a maximal ideal of $R$. Assume first that $R/M$ is a infinite field.  There exists $\alpha\in S$ such that $S=R[\alpha]$ \cite[Theorem 3.8]{ADM}, so  that $R\subset S$ is minimal. Assume now that $ R\subset S$ is t-closed. Then, $M$ is a maximal ideal of $S$  \cite[Lemme 3.10]{Pic 1}. It follows that $R/M\subset S/M$ is an FIP field extension, for which  the Primitive Element Theorem holds. Therefore,  $S=R[\alpha]$ for some $\alpha \in S$, so that $R\subset S$ is   minimal. 
\end{proof}

\subsection{Examples and counterexamples}

The following examples  illustrate Theorem \ref{7.17} and Corollary \ref{7.18} in connection with the FCP or FIP properties, for non minimal extensions.

\begin{example} \label{7.19} In the sequel  $\{X_i\}_{i\in I}$ is a set of indeterminates over a field $k$ with $\mathrm c(k)=2$. 

(1) There exists a pointwise minimal extension $k\subset S$ which is neither a pointwise minimal pair nor an FCP extension. Set $S:=k[\{X_i\}_{i\in I}]/(\{X^2_i\}_{i\in I})$. For each $i\in I$, let $ x_i$ be the class of $X_i$ in $S$. Then, $S$ is a  zero-dimensional local ring with maximal ideal $M:=(\{ x_i\}_{i\in I})$ and $k\subset S$ is  subintegral. Let $x\in M$, there exists a finite set $J\subset I$,  such that $x=\sum_{i\in J}a_ix_i$, with   
 $a_i\in k$ 
for each $i\in J$. Then, $x^2=0$, so that 
$k\subset k[x]$ is  minimal ramified, and 
$k\subset S$ is pointwise minimal  
 but  is not a pointwise minimal pair by Theorem \ref{7.16}(1) (2),  because $x_ix_j\neq 0$ for $i,j\in I,\ i\neq j$. 
If  $|I|=\infty$, then 
$k\subset S$ has not FCP by \cite[Theorem 4.2]{DPP2} 
 and if $|I|<\infty$, then $k\subset S$ has  FCP by the same reference. 
 Moreover, 
 if $|k|=\infty$ and $|I|>2$, then 
 $k\subset S$ has not FIP by Proposition \ref{7.171} since $k$ is infinite and $k\subset S$ is not minimal. 
 Indeed, if $k\subset S$ has  FIP, there exists some $y\in S$ such that $S=k[y]$  (\cite[Theorem 3.8]{ADM}, which would imply that $k\subset S$ is  minimal. 
  If $|k|<\infty$ and $|I|=2$, then $k\subset S$ is  an FIP pointwise minimal extension  which is not a pointwise minimal pair. 

(2) There exists a pointwise minimal pair $k\subset T$ which is neither a co-pointwise minimal extension nor an FCP extension. 
 Set $T:=k[\{X_i\}_{i\in I}]/(\{X^2_i,X_iX_j\}_{i,j\in I})$. For each $i\in I$, let $ x_i$ be the class of $X_i$ in $T$. Then, $T$ is a  zero-dimensional local ring with maximal ideal $M:=(\{ x_i\}_{i\in I})$ and $k\subset T$ is  subintegral. For $x\in M$, there is a finite set $J\subset I$ such that $x=\sum_{i\in J}a_ix_i$, with    $a_i\in k$ for each $i\in J$. Then, $x^2=0$, so that $k\subset k[x]$ is  minimal ramified, and $k\subset T$ is a pointwise minimal  pair by Theorem \ref{7.16}  (2),  because $x_ix_j= 0$ for $i,j\in I$. If  $|I|=\infty$, then $k\subset S$ has not FCP \cite[Theorem 4.2]{DPP2}. If $|I|<\infty$, then $k\subset S$ has  FCP by the same reference. If   $|k|=\infty$ and $|I|>2$, then $k\subset S$ has not FIP  (same reason as in (1)). 

 If $|k|<\infty$ and $|I|=3$, then $k\subset S$ is  an FIP pointwise minimal  pair, but  is not  co-pointwise minimal  by Corollary \ref{7.18} because $\dim_k(M)=3$ since $M=\sum_{i=1}^3kx_i$.   

(3) There exists a co-pointwise minimal extension which is not an  FIP extension. Here, $k$ is  an infinite field. Let $S:=k(X_1,X_2)$ be the field of rational functions over $k$, where $X_1,X_2$ are two indeterminates and set $R:=k(X_1^2,X_2^2)$. Then, $R\subset S$ is a 
 height one radicial  
 field extension of  degree 4 
 because $[S:R]=[k(X_1,X_2):k(X_1^2,X_2)][k(X_1^2,X_2):k(X_1^2,X_2^2)]=4$. Moreover,
 $f^2\in R$ for each $f\in S$. In view of Corollary \ref{7.18}, $R\subset S$ is  a co-pointwise minimal extension. 
 But $R\subset S$ 
 has not  FIP by Proposition \ref{7.171} since $k$ is infinite and $  R \subset S$ is not minimal, for if not its degree would be $2$. 

(4) An FIP co-pointwise minimal extension exists
 by
 Corollary \ref{7.18}(2). 

(5) We give a last example, showing  that  condition (1) of Theorem \ref{7.16}  with $k\neq{}_S^tk={}_S^+k\neq S$  
may occur. Here $k$ is  an infinite field and   
 $k\subset K$  is a minimal radicial  field
  extension of degree 2. For any $y\in K\setminus k$, we have $K=k[y]=k+ky$, with $y^2\in k$. 
 Fix such an $y$ and set $S:=K[X]/(X^2)$, where $X$ is an indeterminate, and let $x$ be the class of $X$ in $S$. Then, $K\subset S$ is  minimal ramified  
 by Lemma \ref{1.3}, 
so that $S$ is a local ring with maximal ideal $M=Kx=kx+kxy$ satisfying $S/M\cong K$, and $S=k[x,y]=k+kx+ky+kxy$. Set $T:=k+M=k+kx+kxy$, so that $S=T[y]=T+Ty=K+M$. It is easy to see that $(T,M)$ is a zero-dimensional local ring. Moreover, $k\subset T$ is subintegral and $T\subset S$ is inert since $k\cong T/M\subset S/M\cong K$ is a minimal  radicial  field extension. To end, $k\subset S/M$ is 
 a height one radicial extension as well as $k\subset S$,
  $M^{[2]}=0$
   and $t^2\in k$ for each $t\in S$ (to see this,   write $t=a+by+x(c+dy)$, with $a,b,c,d\in k$; then, $t^2=a^2+b^2y^2\in k$). Hence, $k\subset S$ satisfies  Theorem \ref{7.16}(1). Moreover, since $x^2=0$, we get that $k\subset k[x]$ is minimal ramified and $k\subset k[y]$ is minimal inert. 
 \end{example}

\begin{remark} \label{7.20} We may remark that an extension $k\subset S$, satisfying  Theorem \ref{7.16}(1),  
 with $k\neq{}_S^tk={}_S^+k\neq S$, 
 has not FIP. Indeed, since $k\subset S/N$ is 
 radicial,   
 $k$ needs to be  an infinite field 
  because any finite extension of a finite field is separable. 
 Then, $k\subset S$ has to be minimal in view of Proposition \ref{7.171}, a contradiction. 
 Moreover,  
 Example \ref{7.19}(5) 
 shows that there exists a pointwise minimal extension $k\subset S$ with $x,y\in S\setminus k$ such that $k\subset k[x]$  and $k\subset k[y]$ are minimal, with  different types. 
  In particular, in this example,  $\ell[k,S]>2$ \cite[Proposition 7.4]{DPPS}.
\end{remark}

We end this section by considering a special situation.  Let $R\subset S$ be a co-pointwise minimal  extension. It follows from Proposition \ref{7.5} that $\ell[R,S]=2$. Therefore, $R\subset T$ and $T\subset S$ are minimal, for any $T\in[R,S]\setminus \{R,S\}$, a  situation studied  by D. E. Dobbs and J. Shapiro (\cite{D} and \cite{DS}). So, let $R\subset T$ and $T\subset S$ be two minimal extensions. Then \cite[Theorem 4.1]{DS}  provides us  13 conditions in order that $R\subset  S$ has FIP and
 \cite[Theorem 2.2]{D} 
 gives 
 2 conditions among 
these 13 conditions  
 in order that $|[R,S]|=3$. Rather than working  with these 13 conditions, we only write in the following proposition the conditions we need in our context.

\begin{proposition}\label{7.21} Let $R\subset T$ and $T\subset S$ be two minimal extensions. The following statements are equivalent:
\begin{enumerate}
\item  $R\subset S$ is an FIP pointwise minimal extension;

\item  $R\subset S$ is an FIP pointwise minimal pair;

\item  $R\subset S$ is an FIP co-pointwise minimal extension; 

\item  $M:= (R:S) \in \mathrm{Max}(R)$,  
$|R/M|< \infty$  and either $S/M\cong(R/M)[X,Y]/(X^2,Y^2, XY)$, or $|R/M|=2$  and  $S/M\cong (R/M)^3$. 
\end{enumerate}

Moreover, if these conditions hold, then $|[R,S]|>3$.
\end{proposition} 

\begin{proof} We have the following implications: (2) $\Rightarrow$ (1) is clear and (3) $\Rightarrow$ (2) by Proposition \ref{7.5}. We also have (4) $\Rightarrow$ (3) by Corollary \ref{7.18}. Indeed, the condition $|R/M|=2$ and $S/M\cong(R/M)^3$ is  Corollary \ref{7.18}(2). Now, assume that $|R/M|< \infty$ and $S/M\cong(R/M)[X,Y]/(X^2,Y^2, XY)$. Set $k:= R/M$ and $S':=S/M$, so that $S'=k[X,Y]/(X^2,Y^2,XY)$. If $x$ and $y$ are the classes of $X$ and $Y$ in $S'$, we get that $S'$ is a zero-dimensional local ring with maximal ideal $N:=kx+ky$ such that $S'=k+N$. It follows that $k\subset S'$ is subintegral since $k\cong S'/N$. Moreover, $N^2=0$ and $\dim_k(N)=2$ imply that $k\subset S'$ is co-pointwise minimal by Corollary \ref{7.18}(1), and so is $R\subset S$. It remains to show  (1) $\Rightarrow$ (4).

Assume that (1) holds. 
Since $R\subset S$ is  pointwise minimal, then $R\subset S$ is either integrally closed, or integral by Proposition  \ref{7.6}. In the first case $R\subset S$ is minimal by Proposition \ref{7.7}, an absurdity. Then, $R\subset S$ is  integral. Let $M$ and $N$ be the respective crucial maximal ideals of $R\subset T$ and $T\subset S$. Now, $M=N\cap R$ by Proposition \ref{7.2} (1), because $M\subseteq N$ is a consequence of $M=(R:S),\ N=(T:S)$ and $N$ lies over a maximal ideal of $R$ 
which is in $\mathrm{Supp}(S/R)=\{M\}$. 
Moreover, since $R\subset S$ has FIP 
and is not minimal, it follows by Proposition \ref{7.171} that 
$R/M$ is a finite field and $R\subset S$ is not a t-closed extension. Let $p:=\mathrm{c}(R/M)$. Therefore, we have to exclude condition  
(1)  of Theorem \ref{7.17}, when $R\subset S$ is not subintegral. 
 
 The remaining conditions of Theorem \ref{7.17} are (1), where $R\subset S$ is  subintegral and (3). In case (1), $R\subset S$ is subintegral and in case (3), $R\subset S$ is seminormal infra-integral, so that 
$R\subset T$ and $T\subset S$ have to be of the same type 
by Proposition \ref{7.90}. 
Hence, we have to consider the conditions of \cite[Theorem 4.1]{DS}, where $ M=N\cap R$, with $R\subset T$ and $T\subset S$ either both ramified or both decomposed, that is to say, conditions (vii) and (xiii) 
of \cite[Theorem 4.1]{DS}. 

Set $k:=R/M,\ T':=T/M,\ N'=N/M$ and $S':=S/M$. If $R\subset T$ and $T\subset S$ are both ramified, with $R\subset S$ a pointwise minimal extension, then $k\subset T'$ and $T'\subset S'$ are both ramified, $k\subset S'$ is a pointwise minimal FIP extension and Theorem \ref{7.16}(1)  holds. There exists $x\in T'$ such that $T'=k[x]$ is a local ring with maximal ideal $N':=kx$ and such that $x^2=0$. Moreover, there exists $y\in S'\setminus T'$, such that $S'=   T'[y]$  is a local ring with maximal ideal $P=kx+ky$, with $y^2,xy\in N'$. Then, $S'=k+kx+ky$ and $y\in P$.  
It follows that  
 $y^2=0$ 
  by Theorem \ref{7.16} (1), 
 and $xy\in N'$, 
 which 
 gives that $xy=ax$, for some $a\in k$. 
 Hence 
 $x(y-a)=0$ in $ S'$. If $a\neq 0$, then $y-a$ is a unit in $S'$, and  $x=0$, a contradiction. 
 This implies that 
 $a=0$ 
 since 
 $\{1,x,y\}$ 
  is 
 a free system over $k$ 
  and 
 $S'\cong k[X, Y]/(X^2,Y^2,XY)$ 
(there is a surjective $k$-algebra morphism $k[X, Y]/(X^2,Y^2,XY)\to S'$  between two vector spaces whose  dimensions are equal). 
 
If $R\subset T$ and $T\subset S$ are both decomposed and  $R\subset S$   is
 pointwise minimal, then $k\subset T'$ and $T'\subset S'$ are both decomposed, and $k\subset S'$ is a pointwise minimal FIP extension. It follows that condition (3) of Theorem \ref{7.16} is satisfied. By \cite[Lemma 5.4]{DPP2}, $S'$ has 3 maximal ideals, whose intersection is $(0)$ and $S'\cong k^3 $, which gives in fact condition (4) of Theorem \ref{7.16} 
since $|k|=2$.

None of the conditions of \cite[Theorem 2.2]{D} holds for a pointwise minimal FIP extension, so that $|[R,S]|>3$. This  can be  easily seen: 
for each $t\in S\setminus T$,  $R\subset R[t]$ is minimal, with $R[t]\neq T,S,R$.   
\end{proof}

\subsection{Transfer properties with respect to Nagata extensions}

We are now looking at the transfer properties of pointwise minimal extensions (resp. pairs) with respect to Nagata rings. To get new results, we consider only non-minimal extensions $R\subset S$ since   $R\subset S$ is minimal if and only if  $R(X)\subset S(X)$ is minimal by   \cite[Theorem 3.4]{DPP3}.

\begin{proposition}\label{7.173} A  non-minimal 
ring extension $R\subset S$ 
 with $M:=(R:S)\in\mathrm{Max}(R)$ 
is a pointwise minimal extension (resp. pair, co-pointwise minimal extension) if and only if $R(X)\subset S(X)$ is a pointwise minimal extension (resp. pair, co-pointwise minimal extension), except for the case where $|R/M|=2$ and $R\subset S$ is a seminormal infra-integral extension.
\end{proposition}

\begin{proof} Observe that $R\subset S$ is  integral (resp. integrally closed)  if and only if so is $R(X)\subset S(X)$ \cite[Proposition 3.8]{DPP3}. Since a pointwise minimal extension (resp. pair, co-pointwise minimal) is either integral or integrally closed by Proposition \ref{7.6}, it is enough to assume that $R\subset S$ is either integral or integrally closed. 
 In the same way, $R\subset S$ is minimal if and only if so is $R(X)\subset S(X)$  \cite[Theorem 3.4]{DPP3}.
Since a pointwise minimal   integrally closed extension is minimal, we  delete this condition.

Now, assume that $R\subset S$ is  integral  
 with $M:= (R:S)\in\mathrm{Max}(R)$. Then $MR(X)S(X)\subseteq R(X)$ and $MR(X)\in\mathrm{Max}(R(X))$ give that $(R(X):S(X))=MR(X)$. 
 Assume also that we have either $|R/M|\neq 2$ or $R\subset S$ is not a seminormal infra-integral extension (see Remark  \ref{7.174}). 

We infer from Theorem \ref{7.17} that $R\subset S$ is  pointwise minimal  if and only if $R\subset S$ satisfies Theorem \ref{7.17}(1). In the same way, $R(X)\subset S(X)$ is  pointwise minimal   if and only if  $R(X)\subset S(X)$ 
satisfies  Theorem \ref{7.17}(1) since $|R(X)/(MR(X))|\neq 2$.  Call     (C)  
 one of the conditions (1) or (2) 
 of Theorem~\ref{7.17}. Then  $R\subset S$ is a pointwise minimal pair if and only if  $R\subset S$ 
satisfies (C) 
 =(1) with ${}_S^tR=R$ or (C)=(2). In the same way, $R(X)\subset S(X)$ is a pointwise minimal pair if and only if $R(X)\subset S(X)$ satisfies (C) 
 =(1) with ${}_{S(X)}^tR(X)=R(X)$ or (C)=(2).

  We are going to show that $R\subset S$ satisfies condition (C) if and only if so does $R(X)\subset S(X)$. Set $k:=R/M$ and $T:=S/M$. Then, $k(X)=R(X)/(MR(X))$ and $T(X)=S(X)/(MR(X))$. 
 
(C)=(2): The equivalence of condition (2)  for $R\subset S$ and $R(X)\subset S(X)$ is immediate. Indeed, $N:=\sqrt[S]{M}\in\mathrm{Max}(S)$ implies $NS(X)=\sqrt[S(X)]{MR(X)}\in\mathrm{Max}(S(X))$. Then, for (C)=(2), $R\subset S$ is a pointwise minimal pair if and only if  $R(X)\subset S(X)$ is a pointwise minimal pair.

(C)=(1): Set $U:={}_S^tR={}_S^+R$. Then, $U(X)={}_{S(X)}^tR(X)={}_{S(X)}^+R(X)$ by \cite[Lemma 3.15]{DPP3}  and ${}_S^tR\subset S$ if and only if ${}_{S(X)}^tR(X)\subset S(X)$. 

  Now $R\subset U$ is subintegral if and only if so is $R(X)\subset U(X)$  \cite[Lemma 3.15]{DPP3}. If these conditions hold, let $N$ be the maximal ideal of $S$, 
  which is also the maximal ideal of $U$
 so that $NS(X)$ is the maximal ideal of $S(X)$. Any element of $S(X)$ (resp. $NS(X)$) is of the form $f(X)=P(X)/Q(X)$, where $P(X)\in S[X]$ (resp. $NS[X]$) and $Q(X)\in \Sigma$. If $R\subset S$ satisfies (C), then 
   $N^{[2]}\subseteq M$. 
   Let $f(X)=P(X)/Q(X)\in NS(X)$, with $P(X)=\sum a_iX^{i}$, where $a_i\in N$ for each $i$. Then, $P(X)^2=\sum a_i^2X^{2i}+2\sum a_ia_jX^{i+j}$. 
    From $(a_i+a_j)^2=a_i^2+a_j^2+2a_ia_j\in M$, with $a_i^2,a_j^2\in M$,  we deduce  
 $2a_ia_j\in M$, so that $f(X)^2\in MR(X)$. The converse is obvious. 
 
The properties of  being  a height one radicial
field extension  and of being of characteristic $p$  are transmitted  to Nagata rings. At last, if $\mathrm{c}(k)=:p$ 
 , then  $N^{[p]}\subseteq k$  if and only if $NT(X)^{[p]}\subseteq k(X)$, in a similar way as it was proved  
  just before. Then $U\subset S$ is a height one radicial extension if and only if so is $U(X)\subset S(X)$.

 To conclude,  $R\subset S $ satisfies  (C) if and only if $R(X)\subset S(X) $ satisfies  (C). Then we have the equivalence of pointwise minimal extension (resp. pair) for $R\subset S$  and $R(X)\subset S(X)$. 
 
 At last, in view of Proposition \ref{7.5}, $R\subset S$ is a co-pointwise minimal extension if and only if $R\subset S$ is a pointwise minimal pair such that $\ell[R,S]=2$. Since $\ell[R,S]=\ell[R(X),S(X)]$ for an FCP extension $R\subset S$ by \cite[Theorem 3.3]{Pic 4} and $R(X)\subset S(X)$ is a co-pointwise minimal extension if and only if $R(X)\subset S(X)$ is a pointwise minimal pair such that $\ell[R(X),S(X)]=2$, we get that $R\subset S$ is a co-pointwise minimal extension if and only if so is
  $R(X)\subset S(X)$.
\end{proof}

\begin{remark}\label{7.174}   In the previous proposition, we had to exclude the cases (3)  and (4) 
of Theorem \ref{7.17}  where $|R/M|=2$ and $R\subset S$ is a seminormal infra-integral extension.   Indeed, in these cases,  
$R(X)/(M(X)) \cong (R/M)(X)$ has  infinitely many elements,  so that $R(X)\subset S(X)$ cannot satisfy conditions (3) or (4) of Theorem \ref{7.17}. 
\end{remark}

\section{Lattices properties of pointwise minimal extensions}

We introduce here FMC  pairs
since we will use them. 

\subsection{FMC pairs} 
An extension $R\subset S$ is 
called an {\it FMC pair} if $R\subset T$ has FMC for each $T\in[R, S]$.
 We  intend  to show  that  FMC pairs are nothing but  FCP extensions.  For  some  results already known, we give   shorter proofs.

 We temporarily introduce a  definition. An extension $U\subseteq V$  is called ${FMC}^\star_n$ if  there is a finite maximal chain from $U$ to $V$ with length $\leq n$ and $U\subseteq \overline U$ has FCP (or equivalently, has FMC). 

\begin{theorem}\label{1.18}  Let $R\subseteq S$ be a ring extension. The following conditions are equivalent:
 \begin{enumerate}

\item  $R\subseteq S$ has FCP; 

\item  There exists a finite maximal chain $\mathcal C$ from $R$ to $S$ with $\overline R \in \mathcal C$;

\item   $R\subset S$ is  an FMC pair;

\item  $R\subset S$ and $R\subseteq \overline R$ have FMC.

\end{enumerate}
\end{theorem}

\begin{proof}  Obviously,  (1) $\Rightarrow$ (2) $\Rightarrow$ (4) and (1) $\Rightarrow$ (3) $\Rightarrow$ (4). It remains to show that (4) $\Rightarrow$ (1).

 So, assume that (4) holds. Then, $R\subseteq \overline R$ has FCP \cite[Theorem 4.2]{DPP2}. 
 
We are going to show by induction on $n$ that  $R\subseteq S$ has FCP, under the above assumption.
 The induction hypothesis is: 
  ${FMC}^\star_n \Rightarrow  FCP$.  
 If $n=1$, then $R\subseteq S$ is minimal and has FCP.

 Assume that the induction hypothesis holds for $n-1$. Let $\mathcal C:=\{R_i\}_{i=0}^n$ be a finite maximal chain with length $n$ such that $R_0=R$ and $R_n=S$. Then, $R\subset R_1$ is  minimal  and $\mathcal C_1:=\{R_i\}_{i=1}^n$ is a finite maximal chain with length $n-1$. 

{\it If $R\subset R_1$ is minimal integral}, then $R_1\subseteq \overline R$, so that $\overline R$ is also the integral closure of $R_1\subseteq S$. Moreover, $R_1\subseteq \overline R$ has FCP. The induction hypothesis gives that $R_1\subseteq S$ has FCP, and so has $\overline R\subseteq S$. To conclude, $R\subseteq S$ has FCP by \cite[Theorem 3.13]{DPP2}.

 {\it If $R\subset R_1$ is Pr\"ufer minimal}, set $N:=\mathcal{C}(R,R_1)$. Then, $N\not\in\mathrm{Supp}_R(\overline R/R) $ in view of 
  Lemmata \ref{1.6} and \ref{1.5}.
  Let $M\in\mathrm{Supp}_R(\overline R/R)$. Then, $M\neq N$ gives that $R_M=(R_1)_M$ and $(\overline R)_M=\overline {(R_M)}$. Set $\mathcal C_M:=\{(R_i)_M\}_{i=0}^n$. Then, $\ell(\mathcal C_M)\leq n-1$ giving that $R_M\subseteq S_M$ has FCP. Let $M\not\in\mathrm{Supp}_R(\overline R/R)$. Then, 
 $(\overline R)_M=R_M$ and $R_M\subseteq S_M$ is an FMC integrally closed extension. It follows that $R_M\subseteq S_M$ has FCP by \cite[Theorem 6.3]{DPP2}. Now from $|\mathrm{Supp}_R(S/R)|<\infty$ \cite[Corollary 3.2]{DPP2}, we infer that $R\subseteq S$ has FCP  \cite[Proposition 3.7 (a)]{DPP2}.
\end{proof}

\begin{remark}\label{1.17}  In the above proposition,  (1)  $\Leftrightarrow$ (2) is proved by Ayache \cite[Theorem 24]{Ay} and   (1) $\Leftrightarrow$ (4) is proved by Ayache and Dobbs  \cite[Theorem 4.12]{AD} in a different way.
\end{remark} 

We are now going to look at the lattice properties of pointwise minimal extensions or pairs. Before, we give the following lemma.
 
 \begin{lemma}\label{7.301}An FMC pair 
 is an affine pair.
\end{lemma} 
\begin{proof}   Obvious, since an FMC extension is of finite type.
\end{proof}

\subsection{ Lattice properties of $[R,S]$}
In the context of a lattice $[R,S]$, some  definitions and properties of lattices have the following formulations.

\begin{enumerate}

 \item If  $R\subseteq S$ has FCP, then $[R,S]$ is a complete Noetherian Artinian lattice for intersection and compositum, whose least element is $R$  and $ S$ is its  largest element. For lattice properties, we use the definitions and  results of 
\cite{We}.  
\item An element $T$ of $[R,S]$ is an {\it atom} (resp.; {\it co-atom}) if and only if $R\subset T$ (resp.; $T\subset S$) is a minimal extension. 
\end{enumerate}

Now $R\subset S$ is called:

(a)  {\it semimodular} if, for each $T_1,T_2\in[R,S]$ such that $T_1\cap T_2\subset T_i$ is minimal for $i=1,2$, then $T_i\subset T_1T_2$ is minimal for $i=1,2$.

(b)  {\it atomistic} if each element of $[R,S]$ is the 
 join (or the least upper bound) of   a set of atoms (see \cite[page 80]{R}). In fact, this is equivalent to  each $T\in[R,S]$ is the  compositum of the atoms contained in $T$.
 
 (c) {\it  finitely atomistic} if, for each $T\in[R,S]$ there exists a finite set $\{A_1,\ldots,A_n\}$ of atoms  such that $T=A_1\cdots A_n$.  
 
(d)   {\it  (finitely) geometric} if it is semimodular and (finitely) atomistic (see \cite[Ex. 7.16, p. 274]{C}). A geometric lattice is sometimes called a {\it matroid} lattice. In another paper to be submitted, we examine  the link between geometric extensions and matroids. 

\begin{remark}\label{proplat}

 i)  If $R\subset S$ has FIP and is geometric, all maximal chains between two comparable elements have the same length (the Jordan-
H\"older
chain condition) \cite[Theorem 1, p. 172]{G}. 

ii) If  $R \subset S$ is  finitely atomistic, then  $R\subseteq S$ is an affine pair.
\end{remark}

 \begin{proposition}\label{7.3}   Let $R\subset S$ be a pointwise minimal extension whose set of atoms is $\mathcal A$.  Then, 
\begin{enumerate}
\item    $R\subset S$ is an affine pair if and only if $R\subset S$ is finitely atomistic. 

\item $R\subset S$ has FIP if and only if $|\mathcal A|<\infty$.

 \item  $R\subset S$ is 
  an affine pair  
 and a pointwise minimal  pair  if and only if $R\subset S$ is a finitely geometric FCP extension. 
 \end{enumerate}
\end{proposition}

 \begin{proof} (1) Assume first that $R\subset S$ is  an affine pair. Any $T\in[R,S]$ is of the form $T=R[x_1,\ldots,x_n]=\prod_{i=1}^nR[x_i]$, where each $R\subset R[x_i]$ is  minimal, that is,  $R[x_i]$  is an atom. Therefore, $R\subset S$ is finitely atomistic. The converse is obvious. 
  
(2) Clearly, if $R\subset S$ has FIP,  $|\mathcal A|<\infty$. Conversely, assume that $|\mathcal A|<\infty$ and set $\mathcal A:=\{A_1,\ldots,A_n\}$.  For each $i\in\{1,\ldots,n\}$, there exists $x_i\in A_i$ such that $A_i=R[x_i]$.
For $T\in [R,S]$ and $x\in T\setminus R$,  the extension $R\subset R[x]$ is minimal,  and then $R[x]\in \mathcal A$. Setting  $I:=\{i\in\{1,\ldots,n\}\mid x_i\in T\}$, we get $T=\prod_{i\in I}R[x_i]$. 

 (3) If $R\subset S$ is 
  an affine pair, 
   $R\subset S$ is   finitely atomistic by (1). Hence, $R\subset S$ is  finitely geometric if and only if $[R,S]$ is semimodular. 

Assume that $R\subset S$ is also a pointwise minimal pair, and let $T_1,T_2 \in [R,S]$ be such that $T:=T_1\cap T_2\subset T_i$ is minimal for $i=1,2$. Let $x_i\in T_i\setminus T$ be such that $T_i=T[x_i]$ for $i=1,2$. Fix some $i=1,2$ and let $j\in\{1,2\}\setminus\{i\}$. Then $T_1T_2=T_iT_j=T_iT[x_j]=T_i [x_j]$, giving that $T_i\subset T_1T_2=T_i[x_j]$ is minimal for $i=1,2$, since $x_j\not\in T_i$ and $R\subset S$ is a pointwise minimal pair. Then, $[R,S]$ is semimodular, and finitely geometric. Any $T\in[R,S]$ can be written $T=R[x_1,\ldots,x_n]$ by (1). Setting $R_0:=R$ and $R_i:=R[x_1,\ldots,x_i]$, we can reorder the $x_i$s so that $x_{i+1}\not\in R_i$ for each $i\in\{0,\ldots,n-1\}$ and then   each $ R_i \subset R_{i+1}=R_i[x_{i+1}]$  is minimal. It follows that there exists a maximal finite chain of $R$-subalgebras of $T$, $R=R_0\subset R_1\subset\cdots\subset R_{n-1}\subset R_n=T$.  Finally,  $[R,S]$ is an FMC pair, and then an FCP extension by Proposition \ref{1.18}.

Conversely, assume that $R\subset S$ is a finitely geometric   FCP extensions, so that $[R,S]$ is semimodular. Let $T\in[R,S]$ and $x\in S\setminus T$.  Since $R\subset S$ has  FCP, there exists a maximal finite chain of $R$-subalgebras of $T$, $R=R_0\subset R_1\subset\cdots\subset R_{n-1}\subset R_n=T$, where each  $R_i\subset R_{i+1}$ is  minimal and of course, $x\not\in R_i$. We are going to show by induction on $i$ that $R_i\subset R_i[x]$ is minimal for each $i\in\{0,\ldots,n\}$. The property holds for $i=0$ since $R\subset S$ is  pointwise minimal. Assume that $R_i\subset R_i[x]$ is minimal for some $i\in\{0,\ldots,n-1\}$. Then, $R_i=R_i[x]\cap R_{i+1}$, because $x\notin R_{i+1}$, implies that $
R_{i+1} \subset R_{i+1}R_i[x]=R_{i+1}[x]$ is minimal. Since this property holds for each $i$, we get that $T=R_n\subset T[x]=R_n[x]$ is minimal, and $R\subset S$ is  a 
pointwise minimal  pair and an affine pair 
 since $R\subset S$ is  finitely atomistic.
\end{proof} 

 For  a set $X$ of atoms of  a geometric extension $[R,S]$, we set $T_X:=\prod_{R_\alpha\in X}R_\alpha$. A subset $I\subseteq [R,S]$ is called  {\it independent } 
(\cite[Definition 9, p.167]{G})  
if, for any finite subsets $J,K$ of $I$, we have $T_J\cap T_K=T_{J\cap K}$ . 
  
\begin{proposition}\label{7.31} A pointwise minimal pair $R\subset S$ has FCP if and only if  $R\subset S$  is 
 an affine pair  
and there exists a finite independent set $I$ of atoms such that $S=T_I$. If these conditions hold, $\ell[R,S]=|I|$. 
\end{proposition}

\begin{proof} 
When  $R\subset S$ is 
 an affine pair 
and a  
pointwise minimal pair, $[R,S]$ is semimodular  since geometric
 by  
 Proposition \ref{7.3} (3). 
 
Assume first that $R\subset S$ has FCP. Then, $R\subset S$ is 
 an affine pair 
by Lemma \ref{7.301} and there is a finite independent set $I$ of atoms such that $S=T_I$ by
Proposition \ref{7.3} (1)
 and  
\cite[Theorem 4, p.174]{G} (take a minimal set $I$ of atoms such that $S=T_I$). 

Conversely, assume that there is a finite independent set $I$ of $n$ atoms with $S=T_I$  and that $R\subset S$ is 
 an affine pair. 
 Since $R\subset S$ is 
  an affine pair and a  
  pointwise minimal pair, $[R,S]$ is semimodular. By \cite[Theorem 4, p.174]{G},  $\ell[R,S]=n$ is finite and then $R\subset S$ has FCP.
\end{proof}

\end{document}